\documentclass{amsart}
\usepackage{amsmath,amsthm,amssymb,mathrsfs}
\usepackage{graphicx}
\usepackage{color}

\theoremstyle{definition}
\newtheorem{Thm}{Theorem}[section]
\newtheorem{Def}[Thm]{Definition}
\newtheorem{Lem}[Thm]{Lemma}
\newtheorem{Cor}[Thm]{Corollary}

\newtheorem{Prop}[Thm]{Proposition}
\newtheorem{Rem}[Thm]{Remark}

\newtheorem{Claim}[Thm]{Claim}
\newtheorem{Not}[Thm]{Notation}

\def\N{\mathbb{N}}

\def\area{\mathrm{Area^{rel}}}
\def\areaq{\mathrm{Area}_Q^{\mathrm{rel}}}

\def\calh{\mathcal H}
\def\calk{\mathcal K}
\def\calr{\mathcal R}

\def\calq{\mathcal Q}

\def\arear{{\rm Area^{rel}}}

\def\scrh{\mathscr{H}}
\def\scrk{\mathscr{K}}

\makeatletter
    
    \@addtoreset{equation}{section}
  \makeatother

\title[On relative hyperbolicity for a group]
{On relative hyperbolicity for a group and relative quasiconvexity for a subgroup}
\author[Y. Matsuda]{Yoshifumi Matsuda}
\author[S. Oguni]{Shin-ichi Oguni}
\author[S. Yamagata]{Saeko Yamagata}
\address[Y. Matsuda]{Graduate School of Mathematical Sciences, University of Tokyo, 3-8-1 Komaba, Meguro-ku, Tokyo, 153-8914 Japan}
\email{ymatsuda@ms.u-tokyo.ac.jp}
\address[S. Oguni]{Department of Mathematics, Faculty of Science, Ehime University, 2-5 Bunkyo-cho, Matsuyama, Ehime, 790-8577 Japan}
\email{oguni@math.sci.ehime-u.ac.jp}
\address[S. Yamagata]{Faculty of Education and Human Sciences, Yokohama National University, 240-8501 Yokohama, Japan}
\email{yamagata@ynu.ac.jp}
\thanks{The first author is supported by the Global COE Program at Graduate School of Mathematical Sciences, the University of Tokyo, and Grant-in-Aid for Scientific Researches for Young Scientists (B) (No. 22740034), Japan Society of Promotion of Science.}
\thanks{The second author is supported by Grant-in-Aid for Scientific Researches for Young Scientists (B) (No. 24740045), Japan Society of Promotion of Science.}
\keywords{Relatively hyperbolic groups, relatively quasiconvex subgroups}
\subjclass[2010]{20F67, 20F65}

\begin{document}

\begin{abstract}
We consider two families of subgroups of a group. 
Each subgroup which belongs to one family is contained in some subgroup which belongs to the other family. 
We then discuss relations of relative hyperbolicity for the group with respect to the two families, respectively. 
If the group is supposed to be hyperbolic relative to the two families, respectively, then we consider relations of relative quasiconvexity for a subgroup of the group with respect to the two families, respectively. 
\end{abstract}


\maketitle

\section{Introduction}
Relative hyperbolicity for groups is a generalization of hyperbolicity for groups (see \cite{Gro87}). 
There exist several definitions of relative hyperbolicity for groups which are mutually equivalent for finitely generated groups (\cite{Bow12}, \cite{D-S05}, \cite{Far98}, \cite{Hru10} and \cite{Osi06}). 
Relative quasiconvexity for subgroups is defined and studied by many authors (see for example \cite{Hru10}). 

The aim of this paper is to extend some known results on relative hyperbolicity for groups and on relative quasiconvexity for subgroups to one of the most general cases. 
Indeed, we adopt Osin's definition \cite[Definition 2.35]{Osi06} (see Subsection \ref{def_rh}) of relative hyperbolicity for a group. 
For a definition of relative quasiconvexity for a subgroup, \cite[Definition 1.2]{M-O-Y1} (see Definition \ref{def-qc} and Remark \ref{rem-qc}) is adopted. 

Throughout this paper, we assume that groups are not necessarily countable. 
Unless otherwise stated, $\Lambda$ denotes a set which is not necessarily countable and for each $\lambda\in\Lambda$, we denote by $M_\lambda$ a set which is also not necessarily countable. 

Let $G$ be a group and let $\{H_\lambda \}=\{\, H_\lambda \mid \lambda \in \Lambda \,\}$ be a family of subgroups of $G$. 
For each $\lambda\in\Lambda$, let $\{K_{\lambda,\mu}\}_\mu=\{\, K_{\lambda, \mu} \mid \mu \in M_{\lambda}\, \}$ be a family of subgroups of $H_\lambda$. 
We set $\{K_{\lambda,\mu}\}=\{\, K_{\lambda, \mu}\mid \lambda\in\Lambda, \mu \in M_{\lambda}\,\}$. 

One of our main theorems is the following on relative hyperbolicity for a group. 

\begin{Thm}\label{blowup}
Let $G$ be a group and $\{H_\lambda \}$ a family of subgroups of $G$. 
For each $\lambda\in\Lambda$, let $\{K_{\lambda, \mu} \}_\mu$ be a family of subgroups of $H_\lambda$. 
Then the following conditions are equivalent:
\begin{enumerate}
\item[(i)] $G$ is hyperbolic relative to $\{H_\lambda \}$ and $\{K_{\lambda, \mu}\}$, respectively. 
\item[(ii)] $G$ is hyperbolic relative to $\{K_{\lambda, \mu}\}$ and the following conditions hold:
\begin{enumerate}
\item[(1)] $G$ satisfies Condition ($a$) with respect to $\{H_\lambda\}$ (see Definition \ref{*}). 
\item[(2)] $\{H_\lambda\}$ is almost malnormal in $G$ (see Subsection \ref{def_rh}). 
\item[(3)] For any $\lambda\in\Lambda$, the group $H_\lambda$ is quasiconvex relative to $\{K_{\lambda, \mu}\}$ in $G$. 
\end{enumerate}
\item[(iii)] $G$ is hyperbolic relative to $\{H_\lambda \}$ and for any $\lambda\in\Lambda$, the group $H_\lambda$ is hyperbolic relative to $\{K_{\lambda, \mu} \}_\mu$. 
There exists a finite subset $\Lambda_1$ of $\Lambda$ such that for any $\lambda\in \Lambda\setminus \Lambda_1$, the subgroup $H_\lambda$ is the free product $H_\lambda=\ast_{\mu\in M_\lambda}K_{\lambda,\mu}$. 
\end{enumerate}
When $\Lambda$ is finite, Condition (ii) (resp.\ (iii)) is replaced with the following condition (ii)' (resp.\ (iii)'):
\begin{enumerate}
\item[(ii)'] $G$ is hyperbolic relative to $\{K_{\lambda,\mu}\}$ and the following conditions hold:
\begin{enumerate}
\item[(2)] $\{H_\lambda\}$ is almost malnormal in $G$. 
\item[(3)] For any $\lambda\in\Lambda$, the group $H_\lambda$ is quasiconvex relative to $\{K_{\lambda, \mu}\}$ in $G$. 
\end{enumerate}
\item[(iii)'] $G$ is hyperbolic relative to $\{H_\lambda\}$ and for any $\lambda\in \Lambda$, the group $H_\lambda$ is hyperbolic relative to $\{K_{\lambda,\mu}\}_\mu$. 
\end{enumerate}
\end{Thm}

\begin{Rem}
If $G$ is countable, and both $\Lambda$ and $\bigsqcup_{\lambda\in \Lambda}M_\lambda$ are finite, then the equivalence between Conditions (i) and (ii)' in Theorem \ref{blowup} is proved by Yang \cite[Theorem 1.1]{Yan11} and the implication from (i) (and (ii)') to (iii)' follows from \cite[Theorem 1.1 and Corollary 3.5]{Yan11}. 
If $G$ is finitely generated, and both $\Lambda$ and $\bigsqcup_{\lambda\in \Lambda}M_\lambda$ are finite, Dru\c{t}u and Sapir show that Condition (iii)' in Theorem \ref{blowup} implies Condition (i) in Theorem \ref{blowup} (see \cite[Corollary 1.14]{D-S05}). 
\end{Rem}

If we omit the condition that $\Lambda$ is finite, then the implication from (ii)' to (i) (resp.\ from (ii)' to (iii)') in Theorem \ref{blowup} is not true. 
The following is an example. 
For each $l\in \mathbb N \cup \{0\}$, let $C_{3l+1}$ and $C_{3l+2}$ be groups isomorphic to $\mathbb Z$ and let $C_{3l}$ be a group isomorphic to $\mathbb Z/2\mathbb Z$. 
For each $m, n\in \mathbb N \cup \{0\}$, set $K_{2m}=C_{3m}\ast C_{3m+1}$, $K_{2m+1}=C_{3m+2}$ and $H_n=C_{3n}\ast C_{3n+1}\ast C_{3n+2}\ast C_{3(n+1)}$. 
We put $G=\ast_{n\in \mathbb N\cup \{0\}}K_n$. 
The group $G$ is thereby hyperbolic relative to $\{K_n\}_{n\in\mathbb N \cup \{0\}}$. 
The family $\{H_n\}_{n\in\mathbb N\cup \{0\}}$ is almost malnormal in $G$ and for any $n\in \mathbb N \cup \{0\}$, the group $H_n$ is quasiconvex relative to $\{K_n\}_{n\in\mathbb N\cup \{0\}}$ in $G$. 
However $G$ does not satisfy Condition ($a$) with respect to $\{H_n\}_{n\in\mathbb N\cup \{0\}}$ because of $H_n \cap H_{n+1}=C_{3(n+1)}\cong \mathbb Z / 2\mathbb Z$. 
It follows that $G$ is not hyperbolic relative to $\{H_n\}_{n\in \mathbb N \cup \{0\}}$. 

We give an example such that if one omits the condition that $\Lambda$ is finite, then the implication from (iii)' to (i) (resp.\ from (iii)' to (ii)') in Theorem \ref{blowup} is not true. 
For each $m\in \mathbb N \cup \{0\}$, let $C_m$ be a group isomorphic to $\mathbb Z /2\mathbb Z$ and set $K_{2m}=\mathbb Z \times C_m$ and $K_{2m+1}=\mathbb Z \times C_m$. 
We put $H_m=K_{2m}\ast_{C_m}K_{2m+1}$ and $G=\ast_{m\in \mathbb N \cup \{0\}}H_m$. 
It is clear that $G$ is hyperbolic relative to $\{H_m\}_{m\in \mathbb N\cup \{0\}}$ and for any $m\in \mathbb N\cup \{0\}$, the group $H_m$ is hyperbolic relative to $\{K_{2m}, K_{2m+1}\}$. 
However $G$ does not satisfy Condition ($a$) with respect to $\{K_n\}_{n\in \mathbb N\cup \{0\}}$. 
The group $G$ is hence not hyperbolic relative to $\{K_n\}_{n\in \mathbb N\cup \{0\}}$. 

The following is another our main theorem on relative quasiconvexity for a subgroup. 

\begin{Thm}\label{qc-updown}
Suppose that $G$ is hyperbolic relative to $\{H_\lambda \}$ and also hyperbolic relative to $\{K_{\lambda, \mu}\}$. 
Then for any subgroup $L$ of $G$, the following conditions are equivalent:
\begin{enumerate}
\item[(i)] $L$ is quasiconvex relative to $\{K_{\lambda, \mu}\}$ in $G$.
\item[(ii)] $L$ is quasiconvex relative to $\{H_\lambda \}$ in $G$. 
For any $g\in G$, the family $\{gLg^{-1} \cap H_\lambda\}_{\lambda\in\Lambda}$ of subgroups of $G$ is uniformly quasiconvex relative to $\{K_{\lambda,\mu}\}$ in $G$ (see Definition \ref{urqc} and Remark \ref{rem-qc}). 
\item[(iii)] $L$ is quasiconvex relative to $\{H_\lambda \}$ in $G$. 
For any $\lambda\in\Lambda$ and $g\in G$, the group $gLg^{-1}\cap H_\lambda$ is quasiconvex relative to $\{K_{\lambda,\mu}\}_\mu$ in $H_\lambda$. 
For each $g\in G$, there exists a finite subset $\Lambda_g$ of $\Lambda$ satisfying the following:
For any $\lambda\in \Lambda \setminus \Lambda_g$, the group $gLg^{-1}\cap H_\lambda$ is decomposed into a free product $gLg^{-1}\cap H_\lambda=\ast_{\mu\in M_\lambda}(gLg^{-1}\cap K_{\lambda,\mu})$. 
\end{enumerate}
When $\Lambda$ is finite, Condition (ii) (resp.\ (iii)) is replaced with the following condition (ii)' (resp.\ (iii)'):
\begin{enumerate}
\item[(ii)'] $L$ is quasiconvex relative to $\{H_\lambda\}$ in $G$. 
For any $\lambda\in\Lambda$ and $g\in G$, the group $gLg^{-1}\cap H_\lambda$ is quasiconvex relative to $\{K_{\lambda,\mu}\}$ in $G$. 
\item[(iii)'] $L$ is quasiconvex relative to $\{H_\lambda \}$ in $G$. 
For any $\lambda\in\Lambda$ and $g\in G$, the group $gLg^{-1}\cap H_\lambda$ is quasiconvex relative to $\{K_{\lambda,\mu}\}_\mu$ in $H_\lambda$. 
\end{enumerate}
\end{Thm}

\begin{Rem}
If $G$ is countable, and both $\Lambda$ and $\bigsqcup_{\lambda\in\Lambda}M_\lambda$ are finite, Yang proves that Condition (i) is equivalent to Condition (ii)' (see \cite[Theorem 1.3]{Yan11}) and it is proved that Condition (ii)' is equivalent to Condition (iii)' in \cite[Proposition 5.1]{M-O-Y3}. 
\end{Rem} 

We note that (ii)' and (iii)' in Theorem \ref{qc-updown} are equivalent even if $\Lambda$ is not finite (see Proposition \ref{nest}). 
We give an example such that if we omit the condition that $\Lambda$ is finite, then the implication from (iii)' to (i) in Theorem \ref{qc-updown} is not true. 

For each $n \in \N \cup \{0\}$, let $K_n$ be an infinite cyclic group generated by $a_n$. 
For each $m \in \N \cup \{0\}$, we put $H_m=K_{2m}\ast K_{2m+1}$ and $G=\ast_{n\in \N \cup \{0\}}K_n$. 
It is clear that $G$ is hyperbolic relative to $\{H_m\}_{m\in\N \cup \{0\}}$ (resp.\ $\{K_n\}_{n\in\N \cup \{0\}}$) and for each $m\in\N \cup \{0\}$, the group $H_m$ is hyperbolic relative to $\{\, K_{2m},K_{2m+1}\,\}$. 
For each $m\in \N \cup \{0\}$, we consider an infinite cyclic subgroup $L_m$ of $H_m$ generated by $a_{2m}a_{2m+1}$. 
For each $m\in \N \cup \{0\}$, the subgroup $L_m$ is then quasiconvex relative to $\{\, K_{2m},K_{2m+1}\,\}$ in $H_m$. 
We put $L=\ast_{m\in \N \cup \{0\}}L_m$ which is a subgroup of $G$. 
The group $L$ is quasiconvex relative to $\{H_m\}_{m\in\N \cup \{0\}}$ in $G$. 
However $L$ is not quasiconvex relative to $\{K_n\}_{n\in\N \cup \{0\}}$ in $G$. 
In fact, if $L$ was quasiconvex relative to $\{K_n\}_{n\in\N \cup \{0\}}$ in $G$, then $L$ would be hyperbolic relative to $\emptyset$, that is, $L$ would be a hyperbolic group by Theorem 4.25 in \cite{M-O-Y1} because  for every $n\in \N \cup \{0\}$ and every $g\in G$, the group $L\cap gK_ng^{-1}$ is trivial. 
This contradicts the fact that all hyperbolic groups are finitely generated. 

\medskip
This paper is organized as follows. 
In Section 2, basic terminology is provided and we show Lemma \ref{minimallift} on minimal lifts of quasigeodesics which is a key lemma in this paper. 
In Section 3, several properties of relative quasiconvexity for subgroups are introduced and Theorem \ref{qc-updown} is proved. 
In Section 4, we prove Theorem \ref{blowup}. 

\section{Preliminaries}
In Subsection 2.1, we recall Osin's definition of relative hyperbolicity for groups and a definition of relative quasiconvexity for subgroups (\cite[Definition 1.2]{M-O-Y1}). 
In Subsection 2.2, a minimal lift of a path is defined and some properties are studied by using minimal lifts. 
In Subsection 2.3, we recall the definition of Condition ($a$). 

We fix some notations in this paper. 
\begin{Not}
Let $G$ be a group which is not necessarily countable.  
Let $\{H_\lambda \}$ be a family of subgroups of $G$. 
For each $\lambda\in\Lambda$, $\{K_{\lambda, \mu} \}_\mu$ denotes a family of subgroups of $H_\lambda$. 

For each $\lambda\in \Lambda$, put $\calh_\lambda=H_\lambda\setminus\{1\}$ and $\calh=\bigsqcup_{\lambda\in\Lambda} \calh_\lambda$. 
For each $\lambda\in\Lambda$, we set $\calk_\lambda=\bigsqcup_{\mu\in M_\lambda} (K_{\lambda,\mu}\setminus \{1\})$ and $\calk=\bigsqcup_{\lambda\in\Lambda} \calk_\lambda$. 
We regard $\calh_\lambda$, $\calh$, $\calk_\lambda$ and $\calk$ as sets of letters. 
\end{Not}
\subsection{Notation and terminology}\label{def_rh}

We refer to \cite{Osi06} for details. 
We say that $G$ is {\it generated by a set $X$ relative to $\{H_\lambda\}$} if $G$ is generated by $X\sqcup \calh$. 
In this case, the set $X$ is said to be a {\it relative generating set of $G$ with respect to $\{H_\lambda\}$}. 
In this paper, we assume that a relative generating set is symmetric. 
If $X$ is finite, we say that $G$ is {\it finitely generated by $X$ relative to} $\{H_\lambda\}$ and $X$ is a {\it finite relative generating set of $G$ with respect to} $\{H_\lambda\}$.  
If there exists a finite relative generating set of $G$ with respect to $\{H_\lambda\}$, we simply say that $G$ is finitely generated relative to $\{H_\lambda\}$. 

The inclusion from $H_\lambda$ into $G$ induces the homomorphism
from the free product $F=F(X) \ast (\ast_{\lambda \in \Lambda} H_\lambda)
$ onto $G$, where $F(X)$ is the free group with the basis $X$.
Let $N$ be the kernel of this homomorphism.
We take a set $\calr$ consisting of words over $X \sqcup \calh$. 
If $N$ is equal to the normal closure of a subset in $F$ whose elements are exactly represented by elements in $\calr$, then 
\begin{equation}\label{representation}
\langle \ X, H_\lambda, \lambda\in\Lambda \mid R=1, \ R\in\calr \ \rangle
\end{equation}
is called a {\it relative presentation of $G$ with respect to} $\{H_\lambda\}$. 
In this paper, $\calr$ is assumed to be symmetric and to contain all cyclic shifts of any elements of $\calr$. 
If $G$ has a relative presentation (\ref{representation}) with respect to $\{H_\lambda\}$, and both $X$ and $\calr$ are finite, then the relative presentation (\ref{representation}) with respect to $\{H_\lambda\}$ is called a {\it finite relative presentation of $G$ with respect to} $\{H_\lambda\}$. 
If there exists a finite relative presentation of $G$ with respect to $\{H_\lambda\}$, the group $G$ is said to be {\it finitely presented relative to} $\{H_\lambda\}$. 

The relative presentation (\ref{representation}) of $G$ with respect to $\{H_\lambda\}$ is said to be {\it reduced} if each $R$ of $\mathcal R$ has minimal word length among all words over $X \sqcup \mathcal{H}$ representing the same element of the group $F$. 
Unless otherwise stated, we assume that every relative presentation of a group with respect to a family of subgroups of the group is reduced. 

We assume that $G$ has the relative presentation (\ref{representation}) with respect to $\{H_\lambda\}$. 
Let $W$ be a word over $X\sqcup \calh$. 
The word length of $W$ is denoted by $\|W\|$. 
Suppose that $W$ represents the neutral element $1$ in $G$. 
The word $W$ is then denoted by $W= \prod_{i=1}^{n} f_i^{-1} R_i f_i$ in $F$, where each $i=1,2,\ldots,n$, $f_i\in F$ and $R_i\in\calr$. 
We denote by $\arear(W)$ the smallest number $n$ among all such representations of $W$ as above and call it the {\it relative area of $W$ with respect to} (\ref{representation}). 
A function $f \colon \N \rightarrow \N$ is called a {\it relative isoperimetric function} of $G$ with respect to (\ref{representation}) if $f$ satisfies the following: 
For any $n\in \mathbb N$ and any word $W$ over $X\sqcup \calh$ representing the neutral element $1$ in $G$ with $\|W\|\leq n$, the inequality $\arear(W)\leq f(n)$ holds. 
The smallest relative isoperimetric function of $G$ with respect to (\ref{representation}) is said to be the {\it relative Dehn function} of $G$ with respect to (\ref{representation}). 
We say that the relative Dehn function $\delta$ of $G$ with respect to (\ref{representation}) is {\it well-defined} if for each $n\in\N$, the value $\delta(n)$ is finite. 
The relative Dehn function is not necessarily well-defined (see \cite[Example 2.55]{Osi06} or \cite[p.100]{Osi06a}). 

We say that $G$ is {\it hyperbolic relative to $\{H_\lambda\}$} if $G$ has a finite relative presentation with respect to $\{H_\lambda\}$ and the relative Dehn function of $G$ with respect to the presentation is linear. 
This definition is independent of the choice of a finite relative presentation of $G$ with respect to $\{H_\lambda\}$ (see \cite[Theorem 2.34]{Osi06}). 

Let $\Gamma$ be a graph with the set $V$ (resp.\ $E$) of vertices (resp.\ edges). 
In this paper, a graph means an oriented graph, i.e., every edge is oriented. 
For each edge $e\in E$, we denote the origin and the terminus of $e$ by $e_-$ and $e_+$, respectively. 
The inverse edge of $e$ is denoted by $e^{-1}$. 
Note that $(e^{-1})_-=e_+$ and $(e^{-1})_+=e_-$. 

For each $i=1,2,\ldots,n$, let $e_i$ be an edge of $\Gamma$. 
Let $p=e_1e_2\cdots e_n$ be a sequence of edges satisfying that for any $i=1,2,\ldots, n-1$, the equality $(e_i)_+=(e_{i+1})_-$ holds. 
The sequence $p$ is called a {\it path} from $(e_1)_-$ to $(e_n)_+$ in $\Gamma$. 
The vertices $(e_1)_-$ and $(e_n)_+$ are denoted by $p_-$ and $p_+$, respectively. 
Assuming that the length of every edge of $\Gamma$ is equal to $1$, we regard $\Gamma$ as a metric space. 
The length of $p$ is denoted by $l(p)$. 
For a path $p=e_1e_2\cdots e_n$, we express a path $(e_n)^{-1}\cdots (e_2)^{-1}(e_1)^{-1}$ by $p^{-1}$. 
When $p_-=p_+$, the path $p$ is called a {\it cycle} in $\Gamma$. 

Let $u$ and $v$ be vertices of $\Gamma$. 
A {\it geodesic} from $u$ to $v$ in $\Gamma$ is a path from $u$ to $v$ in $\Gamma$ whose length is minimal in the set of all paths from $u$ to $v$ in $\Gamma$. 

Let $X$ be a relative generating set of $G$ with respect to $\{H_\lambda\}$. 
We define the Cayley graph $\Gamma(G,X\sqcup \calh)$ of $G$ as a graph such that the set of vertices (resp.\ edges) is equal to $G$ (resp.\ $G\times (X\sqcup \calh)$) and each edge $(g,u)\in G \times (X\sqcup \calh)$ satisfies that $(g,u)_-=g$ and $(g,u)_+=gu$. 

The {\it label} of each edge $e=(g,u)\in G\times (X\sqcup \calh)$ is denoted by $u=\phi(e)$. 
The label of the inverse edge $e^{-1}$ of $e$ is $u^{-1}=\phi(e^{-1})=\phi(e)^{-1}$. 
For a path $p=e_1e_2\cdots e_n$ in $\Gamma(G,X\sqcup\calh)$, the {\it label} of $p$ is defined as $\phi(e_1)\phi(e_2)\cdots \phi(e_n)$ and denoted by $\phi(p)$. 
For each path $p$, we denote by $\overline{\phi(p)}$ the element of $G$ represented by $\phi(p)$. 
 
We recall the definition of Condition ($b$) introduced in \cite[Definition 1.1]{M-O-Y1}. 

\begin{Def}\label{**}
Let $L$ be a subgroup of $G$. 
The subgroup $L$ is said to satisfy {\it Condition} ($b$) {\it with respect to $\{H_\lambda\}$} if for any $y_1,y_2\in G$ with $Ly_1\cap Ly_2=\emptyset$, the subset $\scrh(L,y_1,y_2)=\{H_{\lambda} \mid \lambda\in\Lambda, y_1H_{\lambda}\cap Ly_2\ne\emptyset\}$ of $\{H_\lambda\}$ is finite. 
\end{Def}

Let $Y$ be a subset of $G$. 
For any $g,h\in G$, we put
\begin{equation*}
d_Y(g,h)=\min\{\ n \mid g=hy_1y_2\cdots y_n, i=1,2,\ldots,n, y_i\in Y \text{ or } y_i^{-1}\in Y \}.
\end{equation*}
If there exist no elements $y_i \in Y$ nor $y_i^{-1}\in Y$ with $g=hy_1y_2\cdots y_n$, then we set $d_Y(g,h)=\infty$. 
For any $g,h\in G$, we also denote $d_Y(g,h)$ by $|h^{-1}g|_Y$.

For subsets $Y$, $A$ and $B$ of $G$, put
$$
d_Y(A,B)=\min\{d_Y(a,b)\ | \ a\in A, b\in B\}.
$$

Let $G$ be a group which is hyperbolic relative to $\{H_\lambda\}$. 
When $G$ is finitely generated and $\Lambda$ is finite, Osin defines quasiconvexity for subgroups of $G$ relative to $\{H_\lambda\}$ (\cite[Definition 4.9]{Osi06}). 
We recall \cite[Definition 1.2]{M-O-Y1} which is a generalization of the definition when both $G$ and $\Lambda$ are not necessarily countable. 
If $G$ is finitely generated and $\Lambda$ is finite, \cite[Definition 1.2]{M-O-Y1} is equal to \cite[Definition 4.9]{Osi06}. 

\begin{Def}
\label{def-qc}
Let $G$ be a group which is hyperbolic relative to $\{H_\lambda\}$ and let $X$ be a finite relative generating set of $G$ with respect to $\{H_\lambda\}$. 
A subgroup $L$ of $G$ is said to be {\it pre-quasiconvex relative to $\{H_\lambda\}$ in $G$ with respect to $X$} if there exists a finite subset $Y$ of $G$ satisfying the following:
Let $l$ be an element of $L$ and let $p$ be a geodesic from $1$ to $l$ in $\Gamma(G, X\sqcup \calh)$. 
For each vertex $v$ of $p$, there exists a vertex $w$ of $L$ such that $d_Y(v,w)\le 1$.

If $L$ is pre-quasiconvex relative to $\{H_\lambda\}$ in $G$ with respect to $X$ and satisfies Condition ($b$) with respect to $\{H_\lambda\}$, the group $L$ is said to be {\it quasiconvex relative to $\{H_\lambda\}$ in $G$ with respect to $X$}. 
\end{Def}

For a family of subgroups of a group $G$, we define uniform relative quasiconvexity. 

\begin{Def}\label{urqc}
Suppose that $G$ is hyperbolic relative to $\{H_\lambda\}$ and that $X$ is a finite relative generating set of $G$ with respect to $\{H_\lambda\}$. 
Let $\{L_\nu\}_{\nu\in N}$ be a family of subgroups of $G$. 
We say that $\{L_\nu\}$ is {\it uniformly quasiconvex relative to $\{H_\lambda\}$ in $G$ with respect to $X$} if it holds the following conditions:
\begin{itemize}
\item For any $\nu \in N$, the group $L_\nu$ satisfies Condition ($b$) with respect to $\{H_\lambda\}$. 
\item There exists a finite subset $Z$ of $G$ satisfying the following:
Take any $\nu\in N$ and any geodesic $p$ in $\Gamma(G,X\sqcup\calh)$ whose origin and terminus lie in $L_\nu$. 
For any vertex $u$ of $p$, there is a vertex $v$ of $L_\nu$ with $d_Z(u,v)\leq 1$.  
\end{itemize}
 \end{Def}

\begin{Rem}\label{rem-qc}
One can show that the above definitions are independent of the choice of a finite relative generating set of $G$ with respect to $\{H_\lambda\}$ by using relative hyperbolicity (see \cite[Corollary 4.21]{M-O-Y1}). 
The subgroup $L$ is thus simply said to be {\it pre-quasiconvex} (resp.\ {\it quasiconvex}) {\it relative to $\{H_\lambda\}$ in $G$} if there exists a finite relative generating set $X$ of $G$ with respect to $\{H_\lambda\}$ such that $L$ is pre-quasiconvex (resp.\ quasiconvex) relative to $\{H_\lambda\}$ in $G$ with respect to $X$. 
The family $\{L_\nu\}_{\nu\in N}$ is also simply said to be {\it uniformly quasiconvex relative to $\{H_\lambda\}$ in $G$} if there exists a finite relative generating set $X$ of $G$ with respect to $\{H_\lambda\}$ such that $\{L_\nu\}_{\nu\in N}$ is uniformly quasiconvex relative to $\{H_\lambda\}$ in $G$ with respect to $X$. 
\end{Rem}

We recall the definition of almost malnormality. 
The family $\{H_\lambda\}$ is said to be {\it almost malnormal} in $G$ if 
\begin{itemize}
\item for any distinct $\lambda_1,\lambda_2\in\Lambda$ and any $g\in G$, the group $gH_{\lambda_1}g^{-1}\cap H_{\lambda_2}$ is finite; and
\item for any $\lambda\in\Lambda$ and any $g\in G\setminus H_\lambda$, the group $gH_\lambda g^{-1}\cap H_\lambda$ is finite. 
\end{itemize}
If $G$ is hyperbolic relative to $\{H_\lambda\}$, then $\{H_\lambda\}$ is almost malnormal in $G$ by \cite[Proposition 2.36]{Osi06}.

\subsection{Minimal lifts of quasigeodesics}

Let $G$ be a group with the relative presentation (\ref{representation}) with respect to $\{H_\lambda\}$. 
For a word $W$ over $X\sqcup \mathcal{H}$ and each $\lambda\in\Lambda$, a non-trivial subword $V$ of $W$ is called an {\it $H_\lambda$-subword} of $W$ if $V$ consists of letters from $\calh_\lambda$. 
We say that an $H_\lambda$-subword of $W$ is an {\it $H_\lambda$-syllable} if it is not contained in any longer $H_\lambda$-subword of $W$. 
Let $p$ be a path in $\Gamma(G,X\sqcup \calh)$. 
For each $\lambda\in\Lambda$, we say that a subpath $q$ of $p$ is an {\it $H_\lambda$-component} of $p$ if $\phi(q)$ is an $H_\lambda$-syllable of $\phi(p)$.  
A subpath $q$ of $p$ is also said to be an {\it $\calh$-component} of $p$ if there exists an element $\lambda$ of $\Lambda$ such that $q$ is an $H_\lambda$-component of $p$. 

Let $p$ be a path in $\Gamma(G,X\sqcup \calh)$ and for any $\lambda\in\Lambda$, let $q$ and $r$ be non-trivial subpaths of $p$ labeled by words over $\calh_\lambda$. 
If there exists a path $c$ in $\Gamma(G,X\sqcup \calh)$ from a vertex of $q$ to a vertex of $r$ which is labeled by letters from $\calh_\lambda$, then $q$ and $r$ are said to be {\it connected} and $c$ is called a {\it connector}. 
We permit the case where $c$ is trivial. 
If an $H_\lambda$-component $q$ of $p$ is not connected to any other $H_\lambda$-component of $p$, then we say that $q$ is {\it isolated}. 

For each $\lambda\in \Lambda$, let $g$ be an element of $\calh_\lambda$ such that there exists an $H_\lambda$-syllable of some $R\in \calr$ which represents $g$ in $G$. 
We denote by $\Omega_\lambda$ the set of all such elements as $g \in \calh_\lambda$ and put 
\begin{equation}\label{omega}
\Omega=\bigsqcup_{\lambda \in \Lambda}\Omega_\lambda.
\end{equation}

Note that for each $\lambda\in\Lambda$, the set $\Omega_\lambda$ is symmetric because $\calr$ is symmetric. 
The set $\Omega$ is thus symmetric. 
If (\ref{representation}) is a finite relative presentation of $G$ with respect to $\{H_\lambda\}$, then $\Omega$ is a finite set. 

We obtain the following lemma by the same proof as \cite[Lemma 2.27]{Osi06}. 

\begin{Lem}\label{lem_2.27}
Let $G$ be a group having the finite relative presentation (\ref{representation}) with respect to $\{H_\lambda\}$. 
Let $p$ be a cycle in $\Gamma(G,X\sqcup\calh)$ and for each $k=1,2,\ldots,m$ and $\lambda_k \in \Lambda$, let $p_k$ be an isolated $H_{\lambda_k}$-component of $p$. 
We put $M=\max\{\ \|R\| \mid R\in\calr \ \}$. 
Then for any $k=1,2,\ldots,m$, we obtain that $\overline{\phi(p_k)}\in \langle \Omega_{\lambda_k}\rangle$ and $\sum_{k=1}^m |\overline{\phi(p_k)}|_{\Omega_{\lambda_k}}\leq M\cdot \arear(\phi(p))$. 
\end{Lem}

Let $p$ be a path in $\Gamma(G,X\sqcup \calh)$. 
If the length of any $\calh$-component of $p$ is equal to $1$, then $p$ is said to be {\it locally minimal}. 
If any $\calh$-component of $p$ is isolated, we say that $p$ is a path {\it without backtracking}. 
Note that any geodesic in $\Gamma(G,X\sqcup \calh)$ is a locally minimal path without backtracking. 

In order to define a minimal lift of a locally minimal path without backtracking in $\Gamma(G,X\sqcup\calh)$, we show the following lemma (see also \cite[Lemma 3.2]{Yan11}), which is a generalization of \cite[Proposition 2.9]{Osi06}. 

\begin{Lem}\label{2.27'}
Suppose that $G$ has the finite relative presentation (\ref{representation}) with respect to $\{H_\lambda\}$ and that $X$ is also a finite relative generating set of $G$ with respect to $\{K_{\lambda,\mu}\}$. 
Then for any $\lambda\in\Lambda$, the group $H_\lambda$ is generated by $\Omega_\lambda\sqcup \calk_\lambda$, i.e., $H_\lambda$ is finitely generated by $\Omega_\lambda$ relative to $\{K_{\lambda,\mu}\}_\mu$. 
\end{Lem}

\begin{proof}
Note that for each $\lambda\in\Lambda$, the set $\Omega_\lambda$ is finite. 

We define a map $\pi$ from $\Gamma(G, X\sqcup{\calk})$ to $\Gamma(G,X\sqcup{\calh})$ as follows:
$\pi$ is the identity on the set $G$ of vertices and on the set of edges labeled by elements of $X$. 
For any $\lambda\in \Lambda$, any $\mu \in M_\lambda$ and any edge $e$ labeled by an element of $\calk_{\lambda,\mu}$, by using of the inclusion from $K_{\lambda,\mu}$ into $H_\lambda$, $\pi(e)$ is an edge labeled by an element of $\calh_\lambda$. 

For each $\lambda\in\Lambda$, let  $h$ be an element of $H_\lambda\setminus \{1\}$.
We take a geodesic $r$ from $1$ to $h$ in $\Gamma(G, X\sqcup\calk)$ and put $q=\pi(r)$ which is a path in $\Gamma(G, X\sqcup{\calh})$. 
Since $1$ and $h$ are elements of $H_\lambda$, we have an edge $p$ in $\Gamma(G,X\sqcup\calh)$ from $1$ to $h$ with $\phi(p)=h$ (see Figure \ref{fig1}). 

If $p$ is equal to $q$, then $\phi(r)$ is an element of $\calk_\lambda$ by the definition of $\pi$.
We then obtain $\overline{\phi(p)}=\overline{\phi(r)}\in \langle\calk_\lambda\rangle$. 

If $p$ is not equal to $q$, then we consider a cycle $qp^{-1}$ in $\Gamma(G, X\sqcup{\calh})$. 
If $p^{-1}$ is an isolated $H_\lambda$-component of $qp^{-1}$, then $h=\overline{\phi(p)}\in \langle \Omega_\lambda \rangle$ by Lemma \ref{lem_2.27}. 
If $p^{-1}$ is an $H_\lambda$-component of $qp^{-1}$ but not isolated in $qp^{-1}$, then we denote by $q_1, q_2, \ldots, q_m$ all the $H_\lambda$-components of $q$ connected to $p^{-1}$ such that they are arranged on $q$ in this order. 
When $p^{-1}$ is not even an $H_\lambda$-component of $qp^{-1}$, let $q_1,q_2,\ldots ,q_m$ be all the $H_\lambda$-components of $q$ which are mutually connected, placed on $q$ in this order, and $1=(q_1)_-=p_-$ or $h=(q_m)_+=p_+$ holds. 

\begin{figure}[top]
\begin{center}
\includegraphics[height=6cm]{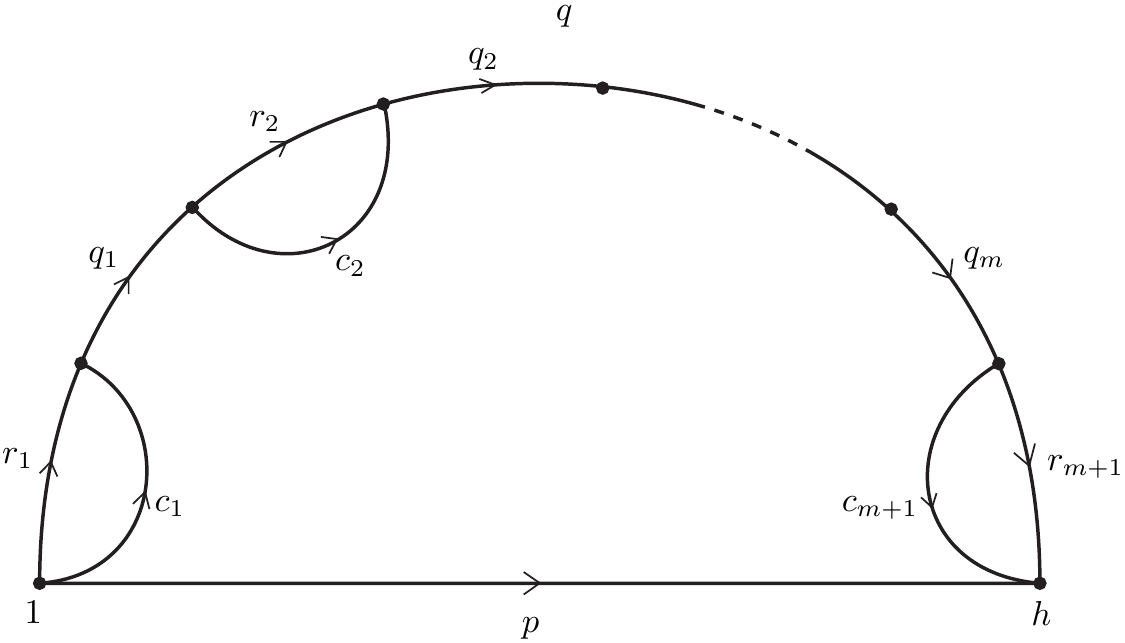}
\end{center}
\caption{A cycle $qp^{-1}$ in $\Gamma(G,X\sqcup \calh)$}
\label{fig1}
\end{figure}

For each $k=1,2,\ldots,m$, we note $\overline{\phi(q_k)}\in\langle \calk_\lambda\rangle$. 
For each $k=2,3,\ldots,m$, let $r_k$ be a subpath of $q$ from $(q_{k-1})_+$ to $(q_k)_-$. 
When $(q_1)_-\ne p_-$ (resp.\ $(q_m)_+\ne p_+$), let $r_1$ (resp.\ $r_{m+1}$) be a subpath of $q$ from $p_-$ to $(q_1)_-$ (resp.\ $(q_m)_+$ to $p_+$). 
If $(q_1)_-=p_-$ (resp.\ $(q_m)_+=p_+$), then we regard as $r_1=\emptyset$ (resp.\ $r_{m+1}=\emptyset$). 
The path $q$ in $\Gamma(G,X\sqcup \calh)$ is thereby represented by $r_1q_1r_2q_2\cdots q_mr_{m+1}$. 

In $\Gamma(G,X\sqcup\calh)$, for each $k=2,3,\ldots,m$, let $c_k$ be a connector from $(q_{k-1})_+$ to $(q_k)_-$. 
When $(q_1)_-\ne p_-$, let $c_1$ be a connector from $p_-$ to $(q_1)_-$ in $\Gamma(G,X\sqcup \calh)$. 
If $(q_1)_-=p_-$, then we regard as $c_1=\emptyset$. 
When $(q_m)_+\ne p_+$, let $c_{m+1}$ be a connector from $(q_m)_+$ to $p_+$ in $\Gamma(G,X\sqcup \calh)$. 
If $(q_m)_+=p_+$, then we regard as $c_{m+1}=\emptyset$. 
We thereby obtain cycles $c_1q_1c_2q_2\cdots q_mc_{m+1}p^{-1}$ and $c_kr_k^{-1}$ in $\Gamma(G, X\sqcup{\calh})$ for each $k=1,2,\ldots,m+1$. 
Since for each $k=1,2,\ldots,m+1$, the $H_\lambda$-component $c_k$ is isolated in $c_kr_k^{-1}$, we obtain $\overline{\phi(c_k)}\in \langle \Omega_\lambda\rangle$ by Lemma \ref{lem_2.27}. 
Because of $h=\overline{\phi(p)}=\overline{\phi(c_1q_1c_2q_2\cdots q_mc_{m+1})}=\overline{\phi(c_1)}\ \overline{\phi(q_1)}\ \overline{\phi(c_2)}\ \overline{\phi(q_2)}\cdots \overline{\phi(q_m)}\ \overline{\phi(c_{m+1})}$, we obtain $h\in \langle \Omega_\lambda\sqcup \calk_\lambda\rangle$.
\end{proof}

In the setting of Lemma \ref{2.27'}, we take a locally minimal path $p$ without backtracking in $\Gamma(G, X\sqcup {\calh})$. 
Note that $\Omega$ is a finite set and $X\sqcup\Omega$ is also a finite relative generating set of $G$ with respect to $\{K_{\lambda,\mu}\}$. 
For each $\lambda\in\Lambda$, by replacing every $H_\lambda$-component $r$ of $p$ with a geodesic in $\Gamma(H_\lambda,\Omega_\lambda\sqcup \calk_\lambda)$ from $r_-$ to $r_+$, we obtain a path $\widehat{p}$ in $\Gamma(G, X\sqcup \Omega\sqcup {\calk})$ from $p_-$ to $p_+$. 
The new path $\widehat p$ is called a {\it minimal lift} of $p$. 
We note that each vertex of $p$ is also one of the vertices of $\widehat{p}$ and the inequality $l(p)\le l(\widehat{p})$ holds. 

Take any vertices $g_1$ and $g_2$ of $\widehat{p}$ and suppose that $g_1$ and $g_2$ are arranged on $\widehat{p}$ in this order. 
We denote by $\widehat{[g_1,g_2]}$ the subpath of $\widehat{p}$ from $g_1$ to $g_2$. 

Let $A$ and $B$ be metric spaces. 
Take two constants $\alpha \geq 1$ and $\beta \geq 0$. 
A map $f \colon A \to B$ is called an $(\alpha,\beta)$-{\it quasi-isometric embedding} if for any $a_1,a_2\in A$, the inequality 
$$
\frac{1}{\,\alpha \,}d(a_1,a_2)-\beta \leq d(f(a_1),f(a_2)) \leq \alpha d(a_1,a_2)+\beta
$$
holds. 
When $A$ is an interval of $\mathbb R$, we call $f$ (and also the image of $f$) an $(\alpha,\beta)$-{\it quasigeodesic} in $B$. 
Let $f$ be an $(\alpha,\beta)$-quasi-isometric embedding from $A$ into $B$. 
We call $f$ an $(\alpha,\beta)$-{\it quasi-isometry} if there exists a constant $k\geq 0$ satisfying that for each $b\in B$, there exists $a\in A$ with $d(f(a),b)\leq k$. 
A map $f \colon A \to B$ is simply called a {\it quasi-isometric embedding} (resp.\ {\it quasi-isometry}) if there exist constants $\alpha\geq 1$ and $\beta\geq 0$ such that the map $f$ is an $(\alpha,\beta)$-quasi-isometric embedding (resp.\ $(\alpha,\beta)$-quasi-isometry). 

The following is a key lemma in this paper. 
This is a generalization of \cite[Lemma 4.4]{MP08} (see another generalization \cite[Proposition 3.9]{Yan11}). 
 
\begin{Lem}\label{minimallift}
Suppose that $G$ has the finite relative presentation (\ref{representation}) with respect to $\{H_\lambda\}$ and $G$ is hyperbolic relative to $\{H_\lambda\}$. 
Suppose that $X$ is also a finite relative generating set of $G$ with respect to $\{K_{\lambda,\mu}\}$. 
Then for any $\alpha\ge 1$ and $\beta\ge 0$, there exist constants $C\ge 1$ and $D\ge 0$ satisfying the following: 
For any locally minimal $(\alpha,\beta)$-quasigeodesic $p$ without backtracking in $\Gamma(G, X\sqcup {\calh})$, a minimal lift $\widehat{p}$ of $p$ is also a locally minimal $(C,D)$-quasigeodesic without backtracking in $\Gamma(G, X\sqcup \Omega\sqcup{\calk})$. 
\end{Lem}

\begin{proof}
Let $\delta$ be the relative Dehn function with respect to (\ref{representation}) and let $B$ be a constant with $\delta(n) \leq Bn$ for any $n\in\N$. 
We put $M=\max\{\,\|R\| \mid R\in\calr\,\}$. 

Let $g$ be an arbitrary element of $G$. 
We denote by $p$ a locally minimal ($\alpha,\beta$)-quasigeodesic from $1$ to $g$ without backtracking in $\Gamma(G,X\sqcup \calh)$. 
Let $\widehat{p}$ be a minimal lift of $p$ in $\Gamma(G,X\sqcup \Omega\sqcup \calk)$. 
We put $\widehat{p}=e_1 e_2 \cdots e_b$, where for each $k=1,2,\ldots, b$, $e_k$ is an edge of $\Gamma(G,X\sqcup \Omega \sqcup \calk)$. 
Set $v_1=1=(e_1)_-, v_2=(e_2)_-, \ldots, v_b=(e_b)_-, v_{b+1}=(e_b)_+=g$. 
We note that the set of vertices of $p$ is contained in that of $\widehat p$ and $\widehat p$ is a locally minimal path without backtracking. 
Put $\gamma=2\alpha+\beta+2$, $C=MB(\alpha+3)+\alpha+1$ and $D=2C+\gamma(MB+1)$.  

Let us take arbitrary $i, j =1,2,\ldots,b+1$ with $i < j$ and put $v_i=g_1$ and $v_j=g_2$. 

\begin{Claim}
The inequality $l(\widehat{[g_1,g_2]})\leq C d_{X\sqcup \Omega\sqcup \calk}(g_1,g_2)+\gamma(MB+1)$ holds. 
\end{Claim}

\begin{proof}
If for some $\lambda\in \Lambda$, both $g_1$ and $g_2$ lie in a minimal lift of an $H_\lambda$-component of $p$, then $[g_1,g_2]$ denotes an edge from $g_1$ to $g_2$ in $\Gamma(G,X\sqcup \calh)$ labeled by an element of $\calh_\lambda$. 
Otherwise, $[g_1,g_2]$ denotes a path from $g_1$ to $g_2$ in $\Gamma(G,X\sqcup \calh)$ as follows:
Let $k_1$ (resp.\ $k_2$) be the minimal (resp.\ maximal) number of $1,2,\ldots,b+1$ with $v_{k_1}\in p\cap\widehat{[g_1,g_2]}$ (resp.\ $v_{k_2}\in p\cap\widehat{[g_1,g_2]}$). 
We put $v_{k_1}=g_3$ and $v_{k_2}=g_4$. 
When $g_1$ (resp.\ $g_2$) is a vertex of $p$, the equality $g_1=g_3$ (resp.\ $g_2=g_4$) holds.  
We denote by $[g_1,g_3]$ (resp.\ $[g_4,g_2]$) an edge from $g_1$ to $g_3$ (resp.\ from $g_4$ to $g_2$) in $\Gamma(G,X\sqcup \calh)$ labeled by an element of $\calh_\lambda$ because for some $\lambda\in\Lambda$, both $g_1$ and $g_3$ (reap.\ $g_4$ and $g_2$) lie in a minimal lift of an $H_\lambda$-component of $p$. 
If $g_1=g_3$ (resp.\ $g_2=g_4$), then we regard as $[g_1,g_3]=\emptyset$ (resp.\ $[g_4,g_2]=\emptyset$). 
The subpath of $p$ from $g_3$  to $g_4$ is denoted by $[g_3,g_4]$. 
If $k_1=k_2$, then we also regard as $[g_3,g_4]=\emptyset$. 
Put $[g_1,g_2]=[g_1,g_3][g_3,g_4][g_4,g_2]$.  
This is a path in $\Gamma(G,X\sqcup \calh)$. 

Whether both $g_1$ and $g_2$ lies in a minimal lift of an $\calh$-component of $p$ or not, the path $[g_1,g_2]$ is a locally minimal ($\alpha,\gamma$)-quasigeodesic without backtracking in $\Gamma(G,X\sqcup \calh)$ by \cite[Lemma 3.5]{Osi06} and $\widehat{[g_1,g_2]}$ is a minimal lift of $[g_1,g_2]$. 

Let $r'$ be a geodesic from $g_1$ to $g_2$ in $\Gamma(G,X\sqcup \Omega\sqcup\calk)$. 
Since for each $\lambda\in\Lambda$ and $\mu\in M_\lambda$, the group $H_{\lambda}$ contains $K_{\lambda,\mu}$ as a subgroup, the following inequality holds; 
\begin{equation*}
\begin{split}
l([g_1,g_2])&\leq \alpha  d_{X\sqcup \calh}(g_1,g_2)+\gamma\\ 
&\leq \alpha d_{X\sqcup \Omega\sqcup\calk}(g_1,g_2)+\gamma=\alpha l(r')+\gamma.
\end{split}
\end{equation*}

Let us regard $r'$ as a path $r$ in $\Gamma(G,X\sqcup \calh)$. 
Note that $l(r)=l(r')$. 
If $r$ is equal to $[g_1,g_2]$, then $r'$ is a minimal lift of $[g_1,g_2]$. 
We then have
\begin{equation*}
l(\widehat{[g_1,g_2]})
=l(r')=d_{X\sqcup \Omega\sqcup\calk}(g_1,g_2).
\end{equation*}

If $r'=\widehat{[g_1,g_2]}$, then $l(\widehat{[g_1,g_2]})=l(r')=d_{X\sqcup\Omega\sqcup\calk}(g_1,g_2)$. 
In the following, we assume that $r\ne[g_1,g_2]$ and $r'\ne\widehat{[g_1,g_2]}$. 

We treat two cases where (i) each $\calh$-component of $[g_1,g_2]^{-1}$ is also an $\calh$-component of $r[g_1,g_2]^{-1}$ and isolated in $r[g_1,g_2]^{-1}$ and (ii) there exists an $\calh$-component of $[g_1,g_2]^{-1}$ such that it is also an $\calh$-component of $r[g_1,g_2]^{-1}$ but not isolated in $r[g_1,g_2]^{-1}$, or it is not even an $\calh$-component of $r[g_1,g_2]^{-1}$. 

(i) Let $h_1,h_2,\ldots,h_n$ be all the $\calh$-components of $[g_1,g_2]$ and for each $k=1,2,\ldots,n$, let $\widehat{h_k}$ be a minimal lift of $h_k$ in $\Gamma(G,X\sqcup\Omega\sqcup \calk)$. 
By Lemma \ref{lem_2.27}, we obtain the inequality
\begin{equation*}
\begin{split}
\sum_{k=1}^{n} l(\widehat{h_k}) &\leq \sum_{k=1}^n |\overline{\phi(h_k)}|_\Omega \leq M\cdot \arear(\phi(r[g_1,g_2]^{-1}))\\
&\leq M\cdot \delta(l(r[g_1,g_2]^{-1}))\leq MB\cdot l(r[g_1,g_2]^{-1})\\
&\leq MB\{ l([g_1,g_2]^{-1})+l(r) \} \leq MB\{ \alpha  l(r')+\gamma+l(r') \}\\
&= MB(\alpha+1)\cdot l(r')+MB\gamma. 
\end{split}
\end{equation*}

By the above inequality, the following inequality holds;
\begin{equation*}
\begin{split}
l(\widehat{[g_1,g_2]})&\leq \sum_{k=1}^n l(\widehat{h_k})+l([g_1,g_2])\\
&\leq  MB(\alpha+1)\cdot l(r')+MB\gamma+\alpha l(r')+\gamma\\
&= \{MB(\alpha+1)+\alpha \}l(r')+\gamma(MB+1)\\
&=\{MB(\alpha+1)+\alpha \}d_{X\sqcup \Omega\sqcup\calk}(g_1,g_2)+\gamma(MB+1)\\
&\leq Cd_{X\sqcup\Omega\sqcup\calk}(g_1,g_2)+\gamma(MB+1).
\end{split}
\end{equation*}

\begin{figure}[top]
\begin{center}
\includegraphics[height=5cm]{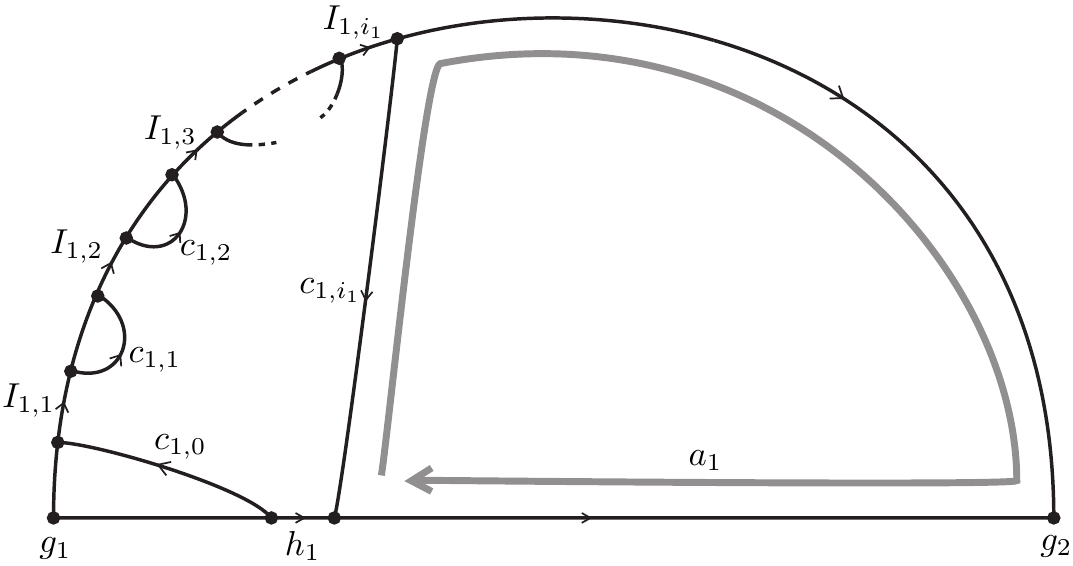}
\end{center}
\caption{Connectors and a cycle $a_1$}
\label{fig2}
\end{figure}

(ii) Let $h_1$ be one of the $\calh$-components of $[g_1,g_2]$ such that $h_1^{-1}$ is also an $\calh$-component of $r[g_1,g_2]^{-1}$ but not isolated in $r[g_1,g_2]^{-1}$, or $h_1^{-1}$ is not even an $\calh$-component of $r[g_1,g_2]^{-1}$. 
We assume that on the subpath of $[g_1,g_2]$ from $g_1$ to $(h_1)_-$, there is no such $\calh$-component of $[g_1,g_2]$. 
Since $[g_1,g_2]$ is a path without backtracking in $\Gamma(G,X\sqcup\calh)$, it does not contain any $\calh$-components which are connected to $h_1$. 
We assume that $h_1$ is an $H_{\lambda_1}$-component of $[g_1,g_2]$. 
If $h_1^{-1}$ is an $H_{\lambda_1}$-component of $r[g_1,g_2]^{-1}$ but not isolated in $r[g_1,g_2]^{-1}$, then the path $r$ in $\Gamma(G,X\sqcup\calh)$ contains all the edges of $r[g_1,g_2]^{-1}$ labeled by elements of $\calh_{\lambda_1}$ which are connected to $h_1$ in $\Gamma(G,X\sqcup \calh)$. 
We denote them by $I_{1,1}, I_{1,2},\ldots, I_{1,i_1}$ and assume that they are placed on $r$ in this order (see Figure \ref{fig2}). 
If $h_1^{-1}$ is not an $H_{\lambda_1}$-component of $r[g_1,g_2]^{-1}$, we assume that $(I_{1,1})_-=(h_1)_-=g_1$ or $(I_{1,i_1})_+=(h_1)_+=g_2$. 

Let $c_{1,0}$ be a connector from $(h_1)_-$ to $(I_{1,1})_-$. 
For each $k=1,2,\ldots, i_1-1$, let $c_{1,k}$ be a connector from $(I_{1,k})_+$ to $(I_{1,k+1})_-$ and let $c_{1,i_1}$ be a connector from $(I_{1,i_1})_+$ to $(h_1)_+$. 
For each $k=1,2,\ldots, i_1-1$, if $(I_{1,k})_+=(I_{1,k+1})_-$, then we regard as $c_{1,k}=\emptyset$. 
If $(I_{1,1})_-=(h_1)_-$ (resp.\ $(I_{1,i_1})_+=(h_1)_+$), then we regard as $c_{1,0}=\emptyset$ (resp.\ $c_{1,i_1}=\emptyset$). 

Let us take (possibly empty) subpaths $[(h_{1})_+,g_2]\subset [g_1,g_2]$ and $[(I_{1,i_1})_+,g_2]\subset r$ and set 
$$
a_1=c_{1,i_1}^{-1}[(I_{1,i_1})_+,g_2][(h_{1})_+,g_2]^{-1}.
$$ 
We assume that a cycle $a_1$ in $\Gamma(G,X\sqcup\calh)$ contains an $\calh$-component of $[(h_1)_+,g_2]^{-1}$ such that it is also an $\calh$-component of $a_1$ but not isolated in $a_1$, or it is not even an $\calh$-component of $a_1$. 
We then take an $\calh$-component $h_2$ of $[(h_1)_+,g_2]$ such that on the subpath of $[(h_1)_+,g_2]$ from $(h_1)_+$ to $(h_2)_-$, there is no such $\calh$-component of $[(h_{1})_+,g_2]$. 

If $h_2^{-1}$ is an $H_{\lambda_2}$-component of $a_1$ but not isolated in $a_1$, the subpaths $I_{2,1}, I_{2,2}$,$\ldots$, $I_{2,i_2}$ of $[(I_{1,i_1})_+,g_2]$ mean all the edges of $a_1$ labeled by elements of $\calh_{\lambda_2}$ which are connected to $h_2$. 
They are arranged on $[(I_{1,i_1})_+,g_2]$ in this order (see Figure \ref{fig3}). 
When $h_2$ is not even an $H_{\lambda_{2}}$-component of $a_1$, we assume that $(I_{2,1})_-=(h_2)_-$ or $(I_{2,i_2})_+=(h_2)_+$. 

We denote by $c_{2,0}$ a connector from $(h_2)_-$ to $(I_{2,1})_-$. 
For each $k=1,2,\ldots,i_2-1$, $c_{2,k}$ means a connector from $(I_{2,k})_+$ to $(I_{2,k+1})_-$ and $c_{2,i_2}$ denotes a connector from $(I_{2,i_2})_+$ to $(h_2)_+$. 
For each $k=1,2,\ldots,i_2-1$, if $(I_{2,k})_+=(I_{2,k+1})_-$, then we regard as $c_{2,k}=\emptyset$. 
When $(I_{2,1})_-=(h_2)_-$ (resp.\ $(I_{2,i_2})_+=(h_2)_+$) holds, we assume that $c_{2,0}=\emptyset$ (resp.\ $c_{2,i_2}=\emptyset$). 
%
\begin{figure}[top]
\begin{center}
\includegraphics[height=5cm]{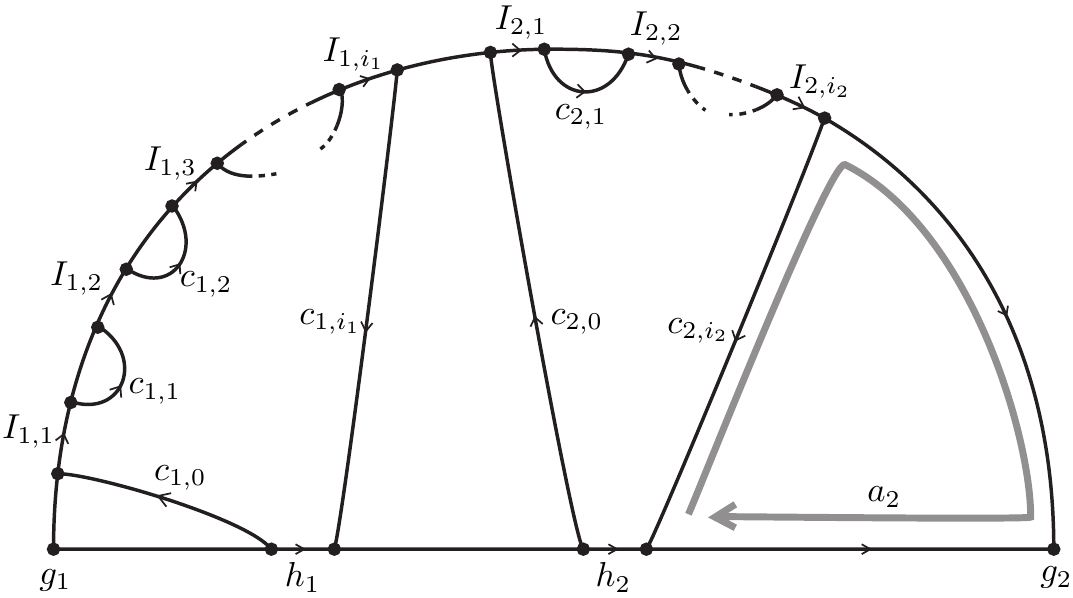}
\end{center}
\caption{A cycle $a_2$}
\label{fig3}
\end{figure}
%
Take (possibly empty) subpaths $[(h_2)_+,g_2]\subset [g_1,g_2]$ and $[(I_{2,i_2})_+,g_2]\subset r$ and set 
$$
a_2=c_{2,i_2}^{-1}[(I_{2,i_2})_+,g_2][(h_{2})_+,g_2]^{-1}. 
$$

We repeat this process $n$ times, and put $[(h_n)_+,g_2]\subset [g_1,g_2]$, $[(I_{n,i_n})_+,g_2]\subset r$ and 
$$
a_n=c_{n,i_n}^{-1}[(I_{n,i_n})_+,g_2][(h_{n})_+,g_2]^{-1}
$$
(see Figure \ref{fig4}). 
We stop the process when every $\calh$-component of $[(h_n)_+,g_2]$ is also an $\calh$-component of $a_n$ and isolated in $a_n$.

For each $k=1,2,\ldots, n$ and $j=0,1,\ldots,i_k$, we denote by $\widehat{c_{k,j}}$ a minimal lift of $c_{k,j}$. 

\begin{figure}[top]
\begin{center}
\includegraphics[height=5cm]{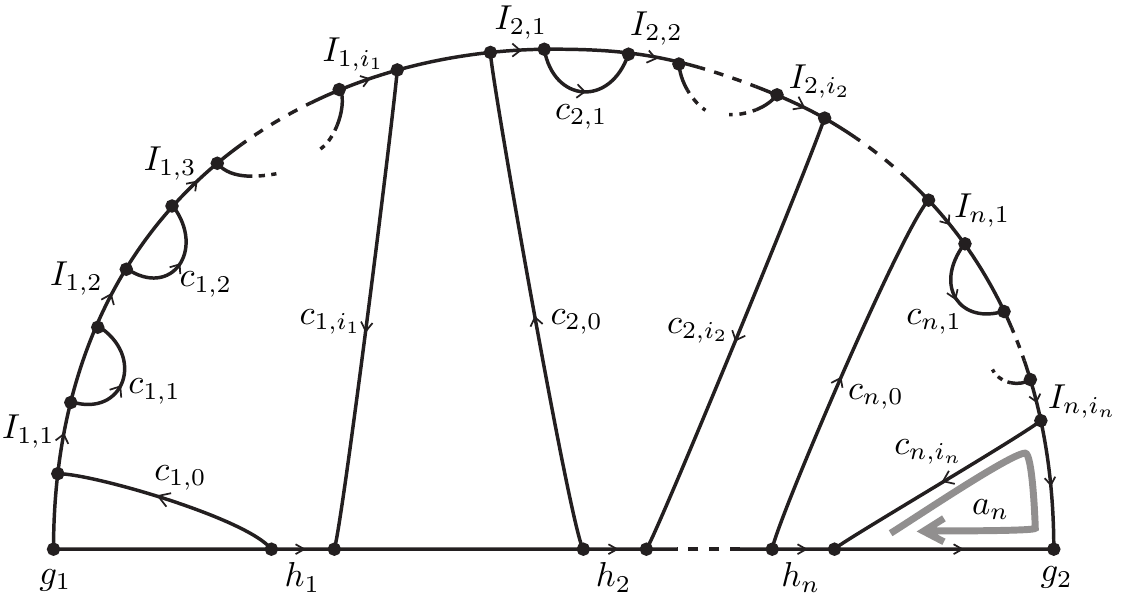}
\end{center}
\caption{A cycle $a_n$}
\label{fig4}
\end{figure}

For each $k=1,2,\ldots,n-1$, we denote the (possibly empty) subpath from $(h_k)_+$ to $(h_{k+1})_-$ of $[g_1,g_2]$ by $[(h_k)_+,(h_{k+1})_-]$. 
The (possibly empty) subpath from $g_1$ to $(h_1)_-$ (resp.\ from $(h_n)_+$ to $g_2$) of $[g_1,g_2]$ is denoted by $[g_1,(h_1)_-]$ (resp.\ $[(h_n)_+,g_2]$). 
For each $k=1,2,\ldots,n$ and $j=1,2,\ldots,i_k-1$, we express the (possibly empty) subpath from $(I_{k,j})_+$ to $(I_{k,j+1})_-$ of $r$ by $[(I_{k,j})_+,(I_{k,j+1})_-]$. 
The (possibly empty) subpath from $g_1$ to $(I_{1,1})_-$ (resp.\ from $(I_{n,i_n})_+$ to $g_2$) of $r$ is denoted by $[g_1,(I_{1,1})_-]$ (resp.\ $[(I_{n,i_n})_+,g_2]$). 
For each $k=1,2,\ldots,n-1$, $[(I_{k,i_k})_+,(I_{k+1,1})_-]$ also denotes the (possibly empty) subpath from $(I_{k,i_k})_+$ to $(I_{k+1,1})_-$ of $r$. 

We define a finite set of cycles in $\Gamma(G,X\sqcup \calh)$ as 
\begin{equation*}
\begin{split}
&\{r_k \mid k=1,2,\ldots,t\}\\
&=\{\,[g_1,(h_1)_-]c_{1,0}[g_1,(I_{1,1})_-]^{-1}\,\}\\
&\cup\{\,c_{k,j}[(I_{k,j})_+,(I_{k,j+1})_-]^{-1} \mid k=1,2,\ldots,n, j=1,2,\ldots,i_k-1\,\} \\
&\cup \{\,[(h_k)_+,(h_{k+1})_-]c_{k+1,0}[(I_{k,i_k})_+,(I_{k+1,1})_-]^{-1}c_{k,i_k}\,\mid k=1,2,\ldots,n-1\}\\
&\cup\{\,[(h_n)_+,g_2][(I_{n,i_n})_+,g_2]^{-1}c_{n,i_n}\,\}
\end{split}
\end{equation*}
(see Figure \ref{fig5}). 
For any $i=1,2,\ldots,t$, if there are $k\in\{\,1,2,\ldots,n\,\}$ and $j\in\{\,1,2,\ldots,i_k\,\}$ such that $r_i$ contains $c_{k,j}$, then $c_{k,j}$ is an isolated $\calh$-component of $r_i$. 
For any $i=1,2,\ldots,m$, denote by $J_i$ the $\calh$-component of $[g_1,g_2]$ satisfying that there is $k=1,2,\ldots,t$ such that $r_k$ contains $J_i$. 
We note that for any $i=1,2,\ldots,m$, $J_i$ is also an isolated $\calh$-component of some $r_k$. 
We note that for any $k=1,2,\ldots,n$ and any $j=1,2,\ldots,i_k$, there is $i\in \{\,1,2,\ldots,t\,\}$ such that $r_i$ contains $c_{k,j}$. 
By Lemma \ref{lem_2.27} and $\sharp\{c_{k , j}\}\leq 2 \sharp \{I_{k,j}\}\leq 2 l(r)$, we obtain the following inequality;
\begin{equation*}
\begin{split}
&\sum_{k=1}^n \sum_{j=0}^{i_k} l(\widehat{c_{k,j}})+\sum_{k=1}^m l(\widehat{J_k})\\
&\le \sum_{k=1}^n \sum_{j=0}^{i_k}|\overline{\phi(c_{k,j})}|_{\Omega}+\sum_{k=1}^m|\overline{\phi(J_k)}|_{\Omega}\\
&\le M\sum_{k=1}^t \arear(\phi(r_k))
\le M\sum_{k=1}^t \delta(l(r_k))
\le MB\sum_{k=1}^t l(r_k)\\
&\le MB \{l([g_1,g_2])+3l(r)\}\le MB\{\alpha l(r')+\gamma+3l(r')\}\\
&=MB(\alpha+3)l(r')+MB\gamma.
\end{split}
\end{equation*}

\begin{figure}[top]
\begin{center}
\includegraphics[height=5cm]{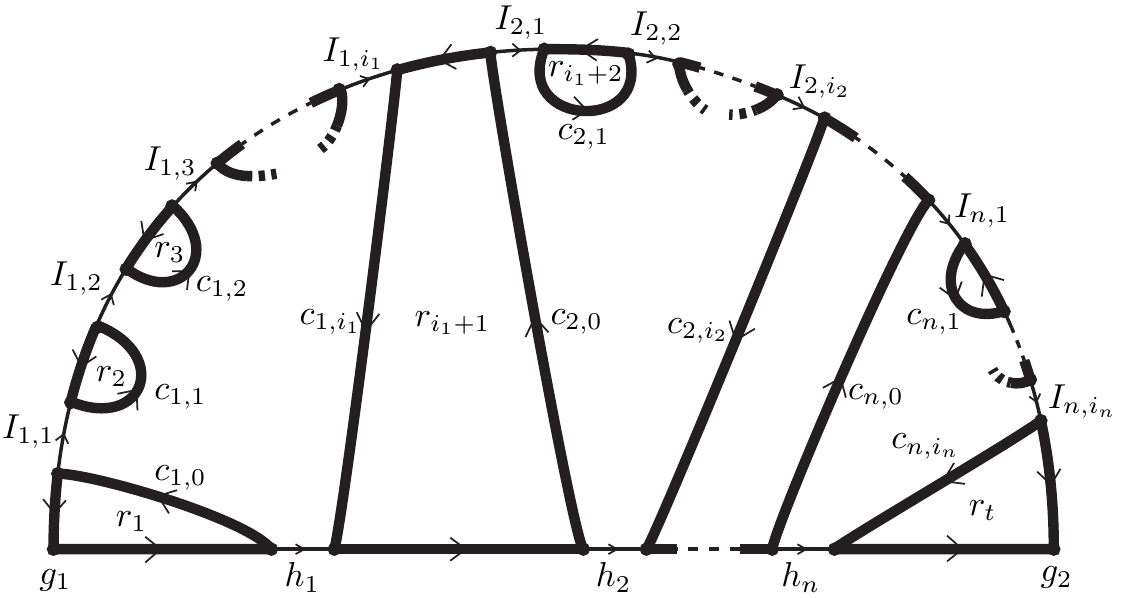}
\end{center}
\caption{Cycles $r_1, \ldots,r_t$}
\label{fig5}
\end{figure}

We denote by $I'_{k, j}$ a subpath of $r'$ which is regarded as a subpath $I_{k, j}$ of $r$.
We set $s=[g_1,g_2]\setminus \left\{\left(\bigsqcup_{k=1}^n h_k\right)\cup \left( \bigsqcup_{k=1}^m J_k \right)\right\}$ and take a path $s \cup \left(\bigsqcup_{k,j} \widehat{c_{k,j}}\right) \cup \left(\bigsqcup_{k,j} {I'_{k,j}}\right)\cup \left(\bigsqcup_{k=1}^m \widehat{J_k}\right)$ from $g_1$ to $g_2$ in $\Gamma(G,X\sqcup\Omega\sqcup\calk)$, where $\widehat{J_k}$ is a minimal lift of $J_k$ in $\Gamma(G,X\sqcup \Omega\sqcup \calk)$. 
Let us denote the path by $\widetilde{[g_1,g_2]}$ (see Figure \ref{fig6}). 
The following inequality then holds;
\begin{equation*}
\begin{split}
l(\widehat{[g_1,g_2]})&\le l(\widetilde{[g_1,g_2]})\le l([g_1,g_2])+\sum_{k=1}^n \sum_{j=0}^{i_k}l(\widehat{c_{k,j}})+l(r)+\sum_{k=1}^m l(\widehat{J_k})\\
&\le \alpha l(r')+\gamma+MB(\alpha+3)l(r')+MB\gamma+l(r') \\
&\le\{MB(\alpha+3)+\alpha+1\}l(r')+\gamma(MB+1)\\
&=\{MB(\alpha+3)+\alpha+1\}d_{X\sqcup \Omega\sqcup\calk}(g_1,g_2)+\gamma(MB+1)\\
&=Cd_{X\sqcup\Omega\sqcup{\calk}}(g_1,g_2)+\gamma(MB+1).
\end{split}
\end{equation*}
%
\begin{figure}[top]
\begin{center}
\includegraphics[height=5cm]{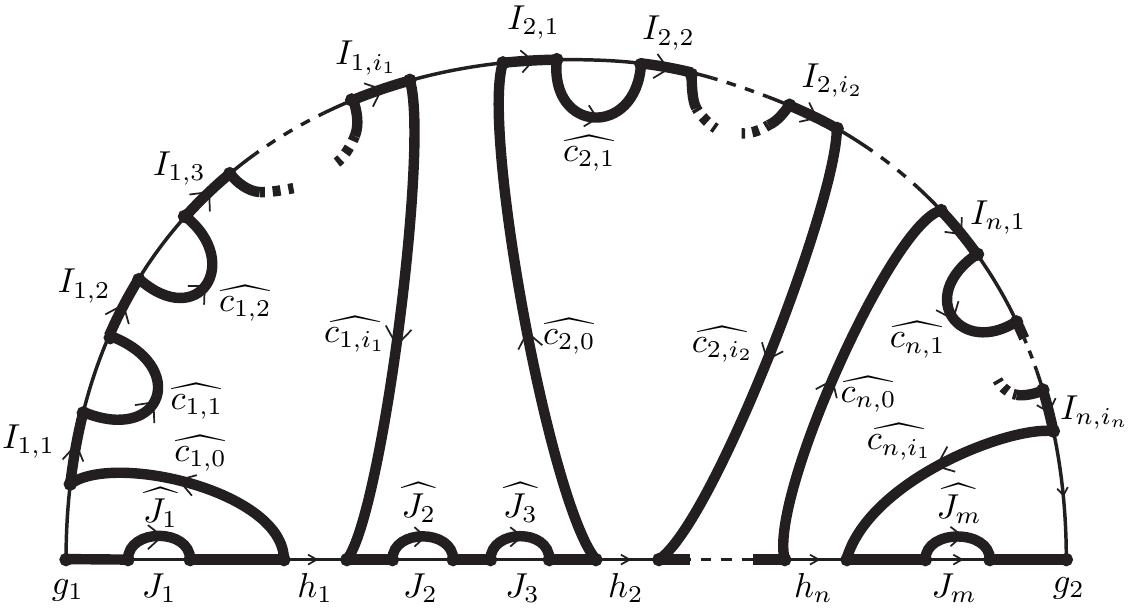}
\end{center}
\caption{A path $\widetilde{[g_1,g_2]}$}
\label{fig6}
\end{figure}
%
\end{proof}

Take arbitrary points $y_1$ and $y_2$ of $\widehat{p}$ and assume that for $i \leq j$, point $y_1$ is in $e_i$ and $y_2$ is in $e_j$. 
If $i=j$, then $y_1$ and $y_2$ are placed on $e_i$ in this order. 
Both $y_1$ and $y_2$ are not necessarily vertices of $\widehat p$. 
The subpath of $\widehat{p}$ from $y_1$ to $y_2$ is denoted by $\widehat{[y_1,y_2]}$. 
We show that $l(\widehat{[y_1,y_2]})\leq Cd_{X\sqcup \Omega \sqcup \calk}(y_1,y_2) + D$. 

Put $g_1=(e_i)_-$ and $g_2=(e_j)_+$. 
We note that $\widehat{[y_1,y_2]} \subset \widehat{[g_1,g_2]}$. 
The following inequality thereby holds;
\begin{equation*}
\begin{split}
l(\widehat{[y_1,y_2]})
&\le l(\widehat{[g_1,g_2]}) \le Cd_{X\sqcup\Omega\sqcup{\calk}}(g_1,g_2)+\gamma(MB+1)\\
&\le C\left(d_{X\sqcup\Omega\sqcup{\calk}}(g_1,y_1)+d_{X\sqcup\Omega\sqcup{\calk}}(y_1,y_2)+d_{X\sqcup\Omega\sqcup{\calk}}(y_2,g_2)\right) +\gamma(MB+1)\\
&\le C\left(d_{X\sqcup\Omega\sqcup{\calk}}(y_1,y_2)+2\right)+\gamma(MB+1)\\
&= Cd_{X\sqcup\Omega\sqcup{\calk}}(y_1,y_2)+2C+\gamma(MB+1).
\end{split}
\end{equation*}
By putting $D=2C+\gamma(MB+1)$, we obtain the assertion. 
\end{proof}

We immediately obtain the following corollary by Lemma \ref{minimallift}.

\begin{Cor}\label{embgraph}
Suppose that $G$ has the finite relative presentation (\ref{representation}) with respect to $\{H_\lambda\}$ and $G$ is hyperbolic relative to $\{H_\lambda\}$. 
We suppose that $X$ is also a finite relative generating set of $G$ with respect to $\{K_{\lambda,\mu}\}$. 
Then for any $\lambda\in\Lambda$, the Cayley graph $\Gamma(H_\lambda, \Omega_\lambda\sqcup{\calk}_\lambda)$ is quasi-isometrically embedded into $\Gamma(G, X\sqcup \Omega\sqcup{\calk})$.
\end{Cor}

By using Corollary \ref{embgraph}, we show the following proposition. 

\begin{Prop}\label{blowup-a}
Suppose that $G$ is hyperbolic relative to $\{H_\lambda\}$ and $\{K_{\lambda,\mu}\}$, respectively. 
Then for any $\lambda\in \Lambda$, the subgroup $H_\lambda$ of $G$ is quasiconvex relative to $\{K_{\lambda,\mu}\}$ in $G$. 
Moreover, for any $\lambda\in\Lambda$, the group $H_\lambda$ is hyperbolic relative to $\{K_{\lambda,\mu}\}_\mu$. 
\end{Prop}

\begin{proof}
The group $G$ is assumed to have the finite relative presentation (\ref{representation}) with respect to $\{H_\lambda\}$. 
Without loss of generality, the set $X$ is assumed to be also a finite relative generating set of $G$ with respect to $\{K_{\lambda,\mu}\}$. 
For each $\lambda\in\Lambda$ and $\mu \in M_\lambda$, we also assume that $H_\lambda$ and $K_{\lambda,\mu}$ are not trivial groups without loss of generality (refer to \cite[Remark 2.2]{M-O-Y1}). 

For any $\lambda\in \Lambda$, it follows from Corollary \ref{embgraph} that $H_\lambda$ is undistorted relative to $\{K_{\lambda,\mu}\}$ in $G$ (see \cite[Definition 4.13]{M-O-Y1}). 
The subgroup $H_\lambda$ of $G$ is hence quasiconvex relative to $\{K_{\lambda,\mu}\}$ in $G$ by \cite[Theorem 1.4 (i)]{M-O-Y1}. 
Moreover, the subgroup $H_\lambda$ is hyperbolic relative to $\{K_{\lambda,\mu}\}_\mu$ by \cite[Theorem 4.25]{M-O-Y1}. 
\end{proof}

\subsection{On Condition ($a$)}

We recall the definition of Condition ($a$) introduced in \cite[Definition 2.11]{M-O-Y1}. 

\begin{Def}\label{*}
Let $G$ be a group with a relative generating set $X$ with respect to $\{H_\lambda\}$. 
The group $G$ is said to satisfy {\it Condition} ($a$) {\it with respect to} $(X, \{H_\lambda\})$ if there exists a finite subset $\Lambda_0$ of $\Lambda$ such that whenever any locally minimal cycle without backtracking in $\Gamma(G,X\sqcup\calh)$ contains an $H_\lambda$-component, the index $\lambda$ is contained in $\Lambda_0$. 
 
If there is a relative generating set $X$ of $G$ with respect to $\{H_\lambda\}$ such that $G$ satisfies Condition ($a$) with respect to $(X,\{H_\lambda\})$, then we simply say that $G$ satisfies {\it Condition} ($a$) {\it with respect to} $\{H_\lambda\}$. 
\end{Def}

\begin{Rem}\label{rem*}
In the setting of Definition \ref{*}, when $\Lambda$ is finite, the group $G$ always satisfies Condition ($a$) with respect to $\{H_\lambda\}$. 
When $G$ has the finite relative presentation (\ref{representation}) with respect to $\{H_\lambda\}$, the group $G$ also satisfies Condition ($a$) with respect to $(X,\{H_\lambda\})$. 
\end{Rem}

\begin{Rem}
If $G$ has a finite relative generating set with respect to $\{H_\lambda\}$, and satisfies Condition ($a$) with respect to $\{H_\lambda\}$, then for any finite relative generating set $Y$ of $G$ with respect to $\{H_\lambda\}$, the group $G$ also satisfies Condition ($a$) with respect to $(Y,\{H_\lambda\})$ by \cite[Lemma 2.14]{M-O-Y1}. 
\end{Rem}

\begin{Not}\label{not*}
Let $X$ (resp.\ $Y$) be a relative generating set of $G$ with respect to $\{H_\lambda\}$ (resp.\ $\{K_{\lambda,\mu}\}$). 
When $G$ satisfies Condition ($a$) with respect to $(X,\{H_\lambda\})$ (resp.\ $(Y,\{K_{\lambda,\mu}\})$), we denote by $\Lambda_X$ (resp.\ $M_Y$) the smallest finite subset of $\Lambda$ (resp.\ $\bigsqcup_{\lambda\in\Lambda}(\{\lambda\}\times M_\lambda)$) such that whenever any locally minimal cycle without backtracking in $\Gamma(G,X\sqcup\calh)$ (resp.\ $\Gamma(G,Y\sqcup \calk)$) contains an $H_\lambda$-component (resp.\ a $K_{\lambda,\mu}$-component), the index $\lambda$ (resp.\ $(\lambda,\mu)$) is contained in $\Lambda_X$ (resp.\ $M_Y$). 
\end{Not}

We introduce Lemma \ref{a1} which is used to show that Condition (ii) implies Condition (i) in Theorem \ref{blowup}. 

\begin{Lem}\label{a1}
Let $G$ be a group having a relative generating set $X$ (resp.\ $Y$) with respect to $\{H_\lambda\}$ (resp.\ $\{K_{\lambda,\mu}\}$).  
Suppose that $G$ satisfies Condition ($a$) with respect to $(X,\{H_\lambda\})$ and $(Y,\{K_{\lambda,\mu}\})$, respectively. 
Then there exists a finite subset $\Lambda_1$ of $\Lambda$ such that $G=\langle\,X\sqcup \left(\bigsqcup_{\lambda\in\Lambda_1}\calh_\lambda\right)\,\rangle \ast\left( \ast_{\lambda\notin\Lambda_1}H_\lambda\right)$ and for any $\lambda\in\Lambda\setminus\Lambda_1$, the equality $H_\lambda=\ast_{\mu\in M_\lambda}K_{\lambda,\mu}$ holds. 
\end{Lem}

\begin{proof}
We take finite subsets $\Lambda_X \subset \Lambda$ and $M_Y \subset \bigsqcup_{\lambda\in \Lambda}(\{\lambda\} \times M_\lambda)$ (see Notation \ref{not*}). 
We denote by $f$ a natural projection from $\bigsqcup_{\lambda\in \Lambda}(\{\lambda\} \times M_\lambda)$ onto $\Lambda$ such that $f((\lambda,\mu))=\lambda$. 
Let us set $\Lambda_1=\Lambda_X \cup f(M_Y)$. 
The set $\Lambda_1$ satisfies the assertion of this lemma. 
\end{proof}

\section{Relative quasiconvexity for a subgroup}

The aim of this section is to prove Theorem \ref{qc-updown}. 
Before proving Theorem \ref{qc-updown}, we show some properties of pre-quasiconvexity for subgroups and results on Condition ($b$). 

Note that for any $\lambda\in\Lambda$, the group $H_\lambda$ is hyperbolic relative to $\{K_{\lambda,\mu}\}_\mu$ in the setting of Lemma \ref{composition} by Proposition \ref{blowup-a}. 

\begin{Lem}\label{composition}
For some $\lambda_0 \in\Lambda$, let $L$ be a subgroup of $H_{\lambda_0}$. 
Suppose that $G$ is hyperbolic relative to $\{H_\lambda\}$ and $G$ is also hyperbolic relative to $\{K_{\lambda,\mu}\}$. 
If the subgroup $L$ is pre-quasiconvex relative to $\{K_{\lambda,\mu}\}$ in $G$, then $L$ is pre-quasiconvex relative to $\{K_{\lambda_0,\mu}\}_\mu$ in $H_{\lambda_0}$
\end{Lem}

\begin{proof}
We assume that $G$ has the finite relative presentation (\ref{representation}) with respect to $\{H_\lambda\}$ and $X$ is also a finite relative generating set of $G$ with respect to $\{K_{\lambda,\mu}\}$. 

By Lemma \ref{2.27'}, for any $\lambda\in\Lambda$, the group $H_\lambda$ is finitely generated by $\Omega_\lambda$ relative to $\{K_{\lambda,\mu}\}_\mu$. 
Take an arbitrary $l\in L\setminus \{1\}$. 
Let $p$ and $q$ be geodesics from $1$ to $l$ in $\Gamma(G,X\sqcup\Omega\sqcup\calk)$ and $\Gamma(H_{\lambda_0},\Omega_{\lambda_0}\sqcup\calk_{\lambda_0})$, respectively. 
By Corollary \ref{embgraph}, there are constants $E\geq 1$ and $F\geq 0$ such that the natural embedding from $\Gamma(H_{\lambda_0}, \Omega_{\lambda_0}\sqcup\calk_{\lambda_0})$ into $\Gamma(G,X\sqcup\Omega\sqcup\calk)$ is an $(E,F)$-quasi-isometric embedding. 
If we regard $q$ as a path in $\Gamma(G,X\sqcup\Omega\sqcup\calk)$, then it is a locally minimal $(E,F)$-quasigeodesic without backtracking in $\Gamma(G,X\sqcup\Omega\sqcup\calk)$. 

By \cite[Theorem 2.14]{MP08}, there exists a finite subset $Y$ of $G$ satisfying the following:
For any locally minimal $(E,F)$-quasigeodesics $p_1$ and $p_2$ without backtracking in $\Gamma(G,X\sqcup\Omega\sqcup\calk)$ with $(p_1)_-=(p_2)_-$ and $(p_1)_+=(p_2)_+$ and any vertex $u_1$ of $p_1$ (resp.\ $v_2$ of $p_2$), there is a vertex $v_1$ of $p_2$ (resp.\ $u_2$ of $p_1$) such that $d_Y(u_1,v_1)\leq 1$ (resp.\ $d_Y(v_2,u_2)\leq 1$). 
For each vertex $v$ of $q$, we thus obtain a vertex $\beta_v$ of $p$ with $d_Y(v,\beta_v)\leq 1$ (see Figure \ref{fig7}). 

\begin{figure}[top]
\begin{center}
\includegraphics[height=4.5cm]{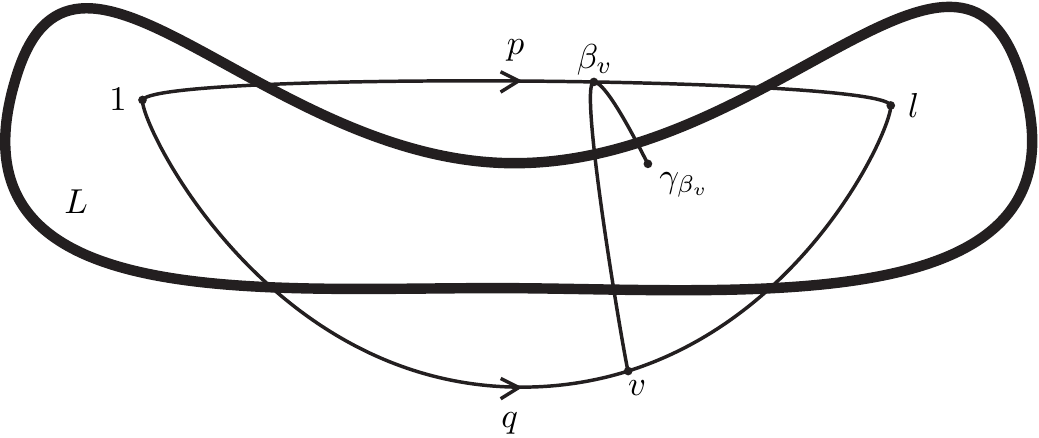}
\end{center}
\caption{$L$ is pre-quasiconvex relative to $\{K_{\lambda,\mu}\}$ in $G$}
\label{fig7}
\end{figure}

If $L$ is pre-quasiconvex relative to $\{K_{\lambda,\mu}\}$ in $G$, then there exists a finite subset $Z$ of $G$ satisfying the following: 
For any geodesic $p_1$ in $\Gamma(G,X\sqcup\Omega\sqcup\calk)$ with $(p_1)_-,(p_1)_+\in L$ and any vertex $u_1$ of $p_1$, there is a vertex $v_1$ of $L$ with $d_Z(u_1,v_1)\leq 1$. 
For each vertex $u$ of $p$, we thereby obtain a vertex $\gamma_u\in L$ with $d_Z(u,\gamma_u)\leq 1$. 
We note that each vertex of $q$ and each vertex of $L$ lie in $H_{\lambda_0}$. 
Put $V=H_{\lambda_0}\cap\{\, g\in G \mid d_{Y\cup Z}(1,g)\leq 2\,\}$. 
The set $V$ is finite. 
For each vertex $v$ of $q$, there are vertices $\beta_v$ of $p$ and $\gamma_{\beta_v}$ of $L$ such that $d_Y(v,\beta_v)\leq 1$ and $d_Z(\beta_v,\gamma_{\beta_v})\leq 1$. 
The inequality $d_V(v,L)\leq d_V(v,\gamma_{\beta_v})\leq 1$ then holds because of $v^{-1}\beta_v\in Y \cup \{1\}$ or $\beta_v^{-1}v\in Y \cup \{1\}$, $\beta_v^{-1}\gamma_{\beta_v}\in Z \cup \{1\}$ or $\gamma_{\beta_v}^{-1}\beta_v\in Z \cup \{1\}$, and $v^{-1}\beta_v\cdot \beta_v^{-1}\gamma_{\beta_v}=v^{-1}\gamma_{\beta_v} \in H_{\lambda_0}$. 
\end{proof}

\begin{Lem}\label{subgr**}
For some $\lambda_0\in\Lambda$, let $Q$ be a subgroup of $H_{\lambda_0}$. 
If $Q$ satisfies Condition ($b$) with respect to $\{K_{\lambda,\mu}\}$ in $G$, then $Q$ satisfies Condition ($b$) with respect to $\{K_{\lambda_0,\mu}\}_\mu$ in $H_{\lambda_0}$. 
\end{Lem}

\begin{proof}
For any $y_1,y_2\in G$ with $Qy_1\cap Qy_2=\emptyset$, we set 
$$
\scrk(Q,y_1,y_2)=\{\, K_{\lambda,\mu} \mid \lambda \in \Lambda, \mu\in M_{\lambda}, y_1K_{\lambda,\mu}\cap Qy_2\ne \emptyset\,\}.
$$ 
For any $y_3,y_4\in H_{\lambda_0}$ with $Qy_3\cap Qy_4=\emptyset$, we put 
$$
\scrk^0(Q,y_3,y_4)=\{\, K_{\lambda_0,\mu} \mid \mu\in M_{\lambda_0}, y_3K_{\lambda_0,\mu}\cap Qy_4\ne \emptyset\,\}.
$$ 
For any $y_3,y_4\in H_{\lambda_0}$ with $Qy_3\cap Qy_4=\emptyset$, it is clear that $\scrk^0(Q,y_3,y_4)\subset \scrk(Q,y_3,y_4)$. 
Because $\scrk(Q,y_1,y_2)$ is finite, the set $\scrk^0(Q,y_3,y_4)$ is also finite. 
\end{proof}

\begin{Prop}\label{nest}
For some $\lambda_0 \in\Lambda$, let $L$ be a subgroup of $H_{\lambda_0}$. 
Suppose that $G$ is hyperbolic relative to $\{H_\lambda\}$ and $\{K_{\lambda,\mu}\}$, respectively. 
Then the following conditions are equivalent:
\begin{enumerate}
\item[(i)]$L$ is quasiconvex relative to $\{K_{\lambda,\mu}\}$ in $G$. 
\item[(ii)]$L$ is quasiconvex relative to $\{K_{\lambda_0,\mu}\}_\mu$ in $H_{\lambda_0}$. 
\end{enumerate}
\end{Prop}

\begin{proof}
The implication from (i) to (ii) follows from Lemmas \ref{composition} and \ref{subgr**}. 

We show the implication from (ii) to (i). 
In the setting of Corollary \ref{embgraph}, $\Gamma(H_{\lambda_0},\Omega_{\lambda_0}\sqcup\calk_{\lambda_0})$ is quasi-isometrically embedded into $\Gamma(G,X\sqcup\Omega\sqcup\calk)$. 
When $L$ is quasiconvex relative to $\{K_{\lambda_0,\mu}\}_\mu$ in $H_{\lambda_0}$, $L$ is undistorted relative to $\{K_{\lambda_0,\mu}\}_\mu$ in $H_{\lambda_0}$ by \cite[Theorem 1.4 (i)]{M-O-Y1}. 
The subgroup $L$ is hence undistorted relative to $\{K_{\lambda,\mu}\}$ in $G$.
By \cite[Theorem 1.4 (i)]{M-O-Y1}, $L$ is also quasiconvex relative to $\{K_{\lambda,\mu}\}$ in $G$. 
\end{proof}

Let $A$ be a subset of $G$. 
For a subset $X$ of $G$ and a constant $k\geq 0$, we put $N_{X,k}(A)=\{g\in G \ | \ d_X(g,A)\leq k\}$. 

\begin{proof}[Proof of Theorem \ref{qc-updown}]
We assume that $G$ has the finite relative presentation (\ref{representation}) with respect to $\{H_\lambda\}$ and $X$ is also a finite relative generating set of $G$ with respect to $\{K_{\lambda,\mu}\}$. 

We prove that Condition (i) implies Condition (iii), Condition (iii) implies Condition (ii), and Condition (i) follows from Condition (ii). 

(i) $\Rightarrow$ (iii) 
We show that $L$ is quasiconvex relative to $\{H_\lambda\}$ in $G$. 

\begin{Claim}\label{claim1}
The group $L$ is pre-quasiconvex relative to $\{H_\lambda\}$ in $G$. 
\end{Claim}

\begin{proof}
By Lemma \ref{minimallift}, there exist constants $C\geq 1$ and $D\geq 0$ satisfying the following:
For any geodesic $p_1$ in $\Gamma(G,X\sqcup\calh)$, a minimal lift $\widehat{p_1}$ of $p_1$ is a locally minimal $(C,D)$-quasigeodesic without backtracking in $\Gamma(G,X\sqcup\Omega\sqcup\calk)$. 

By \cite[Theorem 2.14]{MP08}, there is a finite subset $T$ of $G$ satisfying the following: 
For any two locally minimal $(C,D)$-quasigeodesics $q_1$ and $q_2$ without backtracking in $\Gamma(G,X\sqcup \Omega\sqcup\calk)$ with $(q_1)_-=(q_2)_-$ and $(q_1)_+=(q_2)_+$ and any vertex $u_1$ of $q_1$ (resp.\ $v_2$ of $q_2$), there exists a vertex $v_1$ of $q_2$ (resp.\ $u_2$ of $q_1$) such that $d_T(u_1,v_1)\leq 1$ (resp.\ $d_T(v_2,u_2)\leq 1$). 

Since $L$ is pre-quasiconvex relative to $\{K_{\lambda,\mu}\}$ in $G$, there exists a finite subset $Z$ of $G$ satisfying the following:
For any geodesic $q_1$ in $\Gamma(G,X\sqcup\Omega\sqcup\calk)$ with $(q_1)_-,(q_1)_+\in L$ and any vertex $v_1$ of $q_1$, the inequality $d_Z(v_1,L)\leq 1$ holds. 

Let $l$ be an arbitrary element of $L\setminus \{1\}$ and let $p$ be a geodesic from $1$ to $l$ in $\Gamma(G,X\sqcup\calh)$. 
Let $q$ be a geodesic in $\Gamma(G,X\sqcup\Omega\sqcup\calk)$ from $1$ to $l$. 
For any vertex $u$ of $\widehat{p}$, there is a vertex $v$ of $q$ such that $d_T(u,v)\leq 1$. 

Since the set of vertices of $p$ is contained in that of $\widehat{p}$, we thereby obtain that for any vertex $u$ of $p$, there is a vertex $v$ of $q$ such that $d_{T\cup Z}(u,L)\leq d_T(u,v)+d_Z(v,L)\leq 2$. 
\end{proof}

\begin{Claim}\label{claim2}
The group $L$ satisfies Condition ($b$) with respect to $\{H_\lambda\}$. 
\end{Claim}

\begin{proof}
Let $y_1$ and $y_2$ be arbitrary elements of $G$ with $Ly_1\cap Ly_2=\emptyset$. 
We prove that $\scrh(L,y_1,y_2)=\{\, H_\lambda \mid \lambda \in \Lambda, y_1H_\lambda\cap Ly_2\ne\emptyset\,\}$ is a finite set. 

Take an arbitrary element $H_\lambda$ of $\scrh(L,y_1,y_2)$. 
Let $l$ be an element of $L$ and $h$ be an element of $H_\lambda$ with $y_1h=ly_2$. 
The path in $\Gamma(G,X\sqcup\calh)$ from $y_1$ to $ly_2$ labeled by $h$ is denoted by $s$. 
By Lemma \ref{2.27'}, for any $\lambda\in\Lambda$, the group $H_\lambda$ is finitely generated by $\Omega_\lambda$ relative to $\{K_{\lambda,\mu}\}_\mu$. 
Put $\Lambda_1=\{\lambda \in \Lambda \mid \Omega_\lambda\ne \emptyset\}$ and note that $\Lambda_1$ is finite because $\Omega$ is a finite set. 
For each $\lambda\in\Lambda\setminus \Lambda_1$, the group $H_\lambda$ is finitely generated by $\emptyset$ relative to $\{K_{\lambda,\mu}\}_\mu$. 
In other words, for each $\lambda\in \Lambda\setminus \Lambda_1$, we have $H_\lambda=\ast_{\mu \in M_\lambda}K_{\lambda,\mu}$. 

When $\lambda\in\Lambda\setminus\Lambda_1$, we take a minimal lift $\widehat{s}$ of $s$ and set $\phi(\widehat{s})=k_1k_2\cdots k_n$, where for each $i=1,2,\ldots,n$, the index $\mu_i\in M_\lambda$ and $k_i\in K_{\lambda,\mu_i}\setminus \{1\}$. 
We put $Y_0=\{\,y_1,y_2\,\}\setminus \{1\}$. 
For the constants $C$ and $D$, there are constants $C_1 \geq 1$ and $D_1 \geq 0$ satisfying the following:
For any locally minimal $(C, D)$-quasigeodesic $q_1$ without backtracking in $\Gamma(G,X \sqcup \Omega \sqcup \calk)$, $q_1$ is also a locally minimal $(C_1,D_1)$-quasigeodesic without backtracking in $\Gamma(G,X \sqcup \Omega \sqcup (Y_0\cup Y_0^{-1}) \sqcup \calk)$.
We take a path $r=e_1\widehat{s}e_2$ in $\Gamma(G,X\sqcup\Omega\sqcup (Y_0\cup Y_0^{-1}) \sqcup\calk)$, where $e_1$ is labeled by $y_1$ and $e_2$ is labeled by $y_2^{-1}$. 
By \cite[Lemma 3.5]{Osi06}, the path $r$ is a locally minimal $(C_1,2C_1+D_1+2)$-quasigeodesic without backtracking in $\Gamma(G,X\sqcup\Omega\sqcup (Y_0\cup Y_0^{-1}) \sqcup\calk)$. 
Since $L$ is pre-quasiconvex relative to $\{K_{\lambda,\mu}\}$ in $G$, there is a finite subset $Z_1$ of $G$ satisfying the following by \cite[Theorem 2.14]{MP08}: 
For any locally minimal $(C_1,2C_1+D_1+2)$-quasigeodesic $q_1$ without backtracking in $\Gamma(G,X\sqcup\Omega\sqcup (Y_0\cup Y_0^{-1})\sqcup\calk)$ with $(q_1)_-,(q_1)_+\in L$ and any vertex $u_1$ of $q_1$, the inequality $d_{Z_1}(u_1,L)\leq 1$ holds. 

Let us take a finite set $Y_1=\{\,y \in G \mid d_{Z_1}(1,y)\leq 1 \,\}$. 
Take a finite subset $Y$ of the set $Y_1 \cup \{y_1,y_2\}$ such that $\{y_1,y_2\}\subset Y$, the elements of $Y$ lie in mutually distinct $L$-orbits and $Y_1 \subset LY$. 
We denote by $Y=\{\,y_1,y_2,\ldots,y_m\,\}$ and note that for any distinct $i, j=1,2,\ldots,m$, the equality $Ly_i \cap Ly_j=\emptyset$ holds. 

For each $i=1,2,\ldots, n$, we set $v_i=y_1k_1k_2\cdots k_i$ in $G$. 
Note that $v_n=y_1k_1k_2\cdots k_n=y_1h=ly_2$ in $G$. 
By the construction of $Y$, for each $i=1,2,\ldots,n$, there exist $l_i\in L$ and $y_{j_i} \in Y$ such that $v_i=l_iy_{j_i}$. 
Since $Ly_1\cap Ly_2 = \emptyset$, it follows that there is $i\in \{1,2,\ldots, n\}$ such that $Ly_{j_{i-1}}\cap Ly_{j_i}=\emptyset$. 
We hence obtain that $k_i\in K_{\lambda,\mu_i}\in \scrk(L,y_{j_{i-1}},y_{j_i})=\{\,K_{\lambda_1,\mu_1} \mid \lambda_{1} \in \Lambda, \mu_{1} \in M_{\lambda_{1}}, y_{j_{i-1}}K_{\lambda_1,\mu_1}\cap Ly_{j_i}\ne \emptyset\,\} \subset \bigcup_{i\ne j}^m\scrk(L,y_i,y_j)$. 
We set 
$$
\Lambda_2=\{\,\lambda_1\in \Lambda\setminus \Lambda_1 \mid \text{there exists }\mu\in M_{\lambda_1} \text{ such that } K_{\lambda_1,\mu}\in \bigcup_{i\ne j}^m\scrk(L,y_i,y_j)\,\}.
$$ 
Note that $\Lambda_2$ is finite because $L$ satisfies Condition ($b$) with respect to $\{K_{\lambda,\mu}\}$. 
It follows that $\scrh(L,y_1,y_2)$ is finite because of $\scrh(L,y_1,y_2)\subset \{\,H_\lambda \mid \lambda\in \Lambda_1\cup \Lambda_2\,\}$.  
\end{proof}

By Claims \ref{claim1} and \ref{claim2}, the group $L$ is quasiconvex relative to $\{H_\lambda\}$ in $G$. 

\begin{Claim}\label{claim3}
For any $\lambda\in\Lambda$ and $g\in G$, the subgroup $gLg^{-1}\cap H_\lambda$ is quasiconvex relative to $\{K_{\lambda,\mu}\}_\mu$ in $H_\lambda$. 
\end{Claim}

\begin{proof}
We fix arbitrary $\lambda\in\Lambda$ and $g\in G$. 
By \cite[Corollary 4.24]{M-O-Y1}, the group $gLg^{-1}$ is also quasiconvex relative to $\{K_{\lambda,\mu}\}$ in $G$. 
By Proposition \ref{blowup-a}, the group $H_\lambda$ is quasiconvex relative to $\{K_{\lambda,\mu}\}$ in $G$. 
By \cite[Theorem 1.3]{M-O-Y1}, the group $gLg^{-1}\cap H_\lambda$ is quasiconvex relative to $\{K_{\lambda,\mu}\}$ in $G$. 
It follows from Proposition \ref{nest} that $gLg^{-1}\cap H_\lambda$ is quasiconvex relative to  $\{K_{\lambda,\mu}\}_\mu$ in $H_\lambda$.  
\end{proof}

\begin{Claim}\label{claim4}
For each $g\in G$, there exists a finite subset $\Lambda_g$ of $\Lambda$ satisfying the following:
For any $\lambda\in\Lambda\setminus \Lambda_g$, the group $gLg^{-1}\cap H_\lambda$ is decomposed into a free product $gLg^{-1}\cap H_\lambda =\ast_{\mu\in M_\lambda}(gLg^{-1}\cap K_{\lambda,\mu})$. 
\end{Claim}

\begin{proof}
Since $G$ satisfies Condition ($a$) with respect to $\{H_\lambda\}$ and $\{K_{\lambda,\mu}\}$, respectively, there are finite subsets $\Lambda_X$ of $\Lambda$ and $M_X$ of $\bigsqcup_{\lambda\in\Lambda}(\{\lambda\}\times M_\lambda)$ (see Notation \ref{not*}). 

Set $\Lambda_3=\Lambda_X\cup \{\,\lambda\in \Lambda \mid \text{there exists } \mu\in M_\lambda \text{ such that } (\lambda,\mu)\in M_X\,\}$. 
We note that $G=\langle\, X\sqcup\left(\bigsqcup_{\lambda\in\Lambda_3}\calh_\lambda\right) \,\rangle\ast \left(\ast_{\lambda\notin\Lambda_3}H_\lambda\right)$, where $\langle\, X\sqcup\left(\bigsqcup_{\lambda\in\Lambda_3}\calh_\lambda\right) \,\rangle$ means the subgroup of $G$ generated by $X\sqcup\left(\bigsqcup_{\lambda\in\Lambda_3}\calh_\lambda\right)$, and for any $\lambda\notin\Lambda_3$, the equality $H_\lambda=\ast_{\mu\in M_\lambda}K_{\lambda,\mu}$ holds (see Lemma \ref{a1}). 
By \cite[Lemma 4.16 and Theorem 1.4 (i)]{M-O-Y1}, the group $g L g^{-1}$ is quasiconvex relative to $\{K_{\lambda,\mu}\}$ in $G$. 
For any $g \in G$, there is a finite subset $U_g$ of $G$ satisfying the following:
For any geodesic $q_1$ in $\Gamma(G,X \sqcup \Omega \sqcup \calk)$ with $(q_1)_-, (q_1)_+ \in g L g^{-1}$ and any vertex $u_1$ of $q_1$, $u_1\in gLg^{-1}\cdot U_g$ holds whenever $u_1 \notin g L g^{-1}$. 
 
Without loss of generality, we assume that the elements of $U_g$ lie in mutually distinct $gLg^{-1}$-orbits and $gLg^{-1} \cap U_g =\emptyset$. 
Since $gLg^{-1}$ satisfies Condition ($b$) with respect to $\{K_{\lambda,\mu}\}$ in $G$, for any $u \in U_g$, the set $\scrk(g L g^{-1}, 1, u)=\{\, K_{\lambda,\mu} \mid \lambda \in \Lambda, \mu \in M_{\lambda}, 1 \cdot K_{\lambda,\mu} \cap g L g^{-1} \cdot u \ne \emptyset \,\}$ is a finite subset of $\{K_{\lambda,\mu}\}$. 
It then follows that $\scrk_g = \bigcup_{u \in U_g} \scrk(g L g^{-1}, 1, u)$ is a finite subset of $\{K_{\lambda,\mu}\}$. 

Put $\Lambda_4=\{ \lambda \in \Lambda \mid \text{ there exists } \mu \in M_\lambda \text{ such that } K_{\lambda,\mu} \in \scrk_g \}$ and $\Lambda_g=\Lambda_3 \cup \Lambda_4$. 
We take an arbitrary $\lambda \in \Lambda \setminus \Lambda_g$. 
We note that $\Omega_\lambda= \emptyset$. 
For any $l \in L$ such that $g l g^{-1} \in H_\lambda$, let $q$ be a geodesic in $\Gamma(H_\lambda, \Omega_\lambda \sqcup \calk_\lambda)$ from $1$ to $g l g^{-1}$. 
This is also a geodesic in $\Gamma(G, X \sqcup \Omega \sqcup \calk)$ and the label of every edge of $q$ is in $\calk_{\lambda}$. 
We denote by $1=v_1, v_2, \ldots , v_n=g l g^{-1}$ all the vertices of $q$ placed on $q$ in this order. 
Note that for any $i=1,2,\ldots,n$, the vertex $v_i$ is contained in $H_\lambda$. 
If $v_2$ was not in $g L g^{-1}$, then there would be an element $u$ of $U_g$ such that $v_2 u^{-1} \in g L g^{-1}$.
This is a contradiction because of $\lambda\notin \Lambda_4$. 
It follows that $v_2\in gLg^{-1}$. 
By inductive argument, we obtain that $v_1,v_2,\ldots,v_n\in gLg^{-1}$. 
For the label $\phi(e)$ of each edge $e$ of $q$, there is thus $\mu \in M_\lambda$ such that $\phi(e)\in gLg^{-1}\cap K_{\lambda,\mu}$. 
By \cite[Corollary just after Proposition 3 in Chapter I,  \S 1 (p.6)]{Serre}, we obtain that $gLg^{-1}\cap H_\lambda=\ast_{\mu\in M_\lambda}(gLg^{-1}\cap K_{\lambda,\mu})$. 
\end{proof}

(iii) $\Rightarrow$ (ii)
Since for any $\lambda\in\Lambda$ and any $g\in G$, the group $gLg^{-1}\cap H_\lambda$ is quasiconvex relative to $\{K_{\lambda,\mu}\}$ in $G$ by Proposition \ref{nest}, it therefore satisfies Condition ($b$) with respect to $\{K_{\lambda,\mu}\}$ and there is a finite subset $Z_{\lambda,g}$ of $G$ satisfying the following: 
For any geodesic $q_1$ in $\Gamma(G,X\sqcup\Omega\sqcup\calk)$ with $(q_1)_-,(q_1)_+\in gLg^{-1}\cap H_\lambda$ and any vertex $u_1$ of $q_1$, the inequality $d_{Z_{\lambda,g}}(u_1,gLg^{-1}\cap H_\lambda)\leq 1$ holds. 
Put $\Lambda_5=\Lambda_g\cup \Lambda_X\cup \{\,\lambda\in\Lambda \mid \text{there exists }\mu \in M_\lambda \text{ such that } (\lambda,\mu)\in M_X\,\}$. 
This is a finite subset of $\Lambda$. 

Note that if $\lambda\in \Lambda\setminus \Lambda_5$, then $\Omega_\lambda=\emptyset$, and for any $\lambda\in \Lambda\setminus \Lambda_5$, any geodesic in $\Gamma(H_\lambda,\Omega_\lambda\sqcup\calk_\lambda)$ is also a geodesic in $\Gamma(G,X\sqcup\Omega\sqcup\calk)$ by Lemma \ref{a1}. 

\begin{Claim}\label{claim5}
Take arbitrary $\lambda\in\Lambda\setminus \Lambda_5$ and $l\in L\setminus \{1\}$ with $glg^{-1}\in H_\lambda$. 
Let $q$ be a geodesic in $\Gamma(G,X\sqcup\Omega\sqcup\calk)$ from $1$ to $glg^{-1}$. 
Then any vertex of $q$ is contained in $gLg^{-1}\cap H_\lambda$. 
\end{Claim}

\begin{proof}
Let us denote by $\langle\, X\sqcup \left(\bigsqcup_{\lambda\in\Lambda_5}\calh_\lambda\right) \,\rangle$ the subgroup of $G$ generated by $X\sqcup \left( \bigsqcup_{\lambda\in \Lambda_5} \calh_\lambda \right)$. 
By Lemma \ref{a1}, since $G=\langle\, X\sqcup \left(\bigsqcup_{\lambda\in\Lambda_5}\calh_\lambda\right) \,\rangle\ast (\ast_{\lambda\notin \Lambda_5}H_\lambda)$ and for any $\lambda\in\Lambda\setminus \Lambda_5$, the equality $gLg^{-1}\cap H_\lambda=\ast_{\mu\in M_\lambda}(gLg^{-1}\cap K_{\lambda,\mu})$ holds, we put $glg^{-1}=k_1k_2\cdots k_n$ in $G$, where for each $i=1,2,\ldots,n$, the index $\mu_i\in M_\lambda$ and $k_i\in gLg^{-1}\cap K_{\lambda,\mu_i} \subset gLg^{-1}\cap H_\lambda$. 
Every vertex of $q$ is thus an element of $gLg^{-1}\cap H_\lambda$. 
\end{proof}

By Claim \ref{claim5}, we set $Z_g=\bigcup_{\lambda\in\Lambda_5}Z_{\lambda,g}$. 
For any $g\in G$, the set $Z_g$ then satisfies the following:
Take any $\lambda\in\Lambda$ and any geodesic $q$ in $\Gamma(G,X\sqcup \Omega \sqcup \calk)$ whose origin and terminus lie in $gLg^{-1} \cap H_\lambda$. 
For any vertex $u$ of $q$, there is a vertex $v$ of $gLg^{-1}\cap H_\lambda$ with $d_{Z_g}(u,v)\leq 1$. 

(ii) $\Rightarrow$ (i)
We prove that $L$ is pre-quasiconvex relative to $\{K_{\lambda,\mu}\}$ in $G$ and satisfies Condition ($b$) with respect to $\{K_{\lambda,\mu}\}$. 

\begin{Claim}\label{claim6}
The group $L$ is pre-quasiconvex relative to $\{K_{\lambda,\mu}\}$ in $G$. 
\end{Claim}

\begin{proof}
Since $L$ is pre-quasiconvex relative to $\{H_\lambda\}$ in $G$, there is a finite subset $A$ of $G$ satisfying the following: 
For any geodesic $p_1$ in $\Gamma(G,X\sqcup\calh)$ with $(p_1)_-,(p_1)_+ \in L$ and any vertex $u_1$ of $p_1$, the inequality $d_{A}(u_1,L)\leq 1$ holds. 
Without loss of generality, the set $A$ is assumed to satisfy that if $a_1\ne a_2$, then $La_1\cap La_2=\emptyset$ and not to contain any elements of $L$. 

For any $g\in G$, the family $\{g(g^{-1} L g \cap H_\lambda) g^{-1}\}_{\lambda\in \Lambda}=\{L\cap g H_\lambda g^{-1}\}_{\lambda\in\Lambda}$ of subgroups of $G$ is also uniformly quasiconvex relative to $\{K_{\lambda,\mu}\}$ in $G$ by the proof of \cite[Lemma 4.16 and Theorem 1.4 (i)]{M-O-Y1}. 
For any $g \in G$, there is thus a finite subset $Z_g$ of $G$ satisfying the following:
For any $\lambda\in\Lambda$, any geodesic $q_1$ in $\Gamma(G,X\sqcup\Omega\sqcup \calk)$ with $(q_1)_-,(q_1)_+ \in L\cap g H_\lambda g^{-1}$ and any vertex $u_1$ of $q_1$, the inequality $d_{Z_{g}}(u_1, L)\leq d_{Z_{g}}(u_1, L\cap gH_\lambda g^{-1})\leq 1$ holds. 
We set $Z_A=\bigcup_{a\in A\cup \{1\}} Z_{a}$. 
This is a finite subset of $G$. 

Take 
$$
\Lambda_6=\{\, \lambda\in\Lambda \mid \text{there exist }a_1,a_2 \in A \cup \{1\} \text{ and } h\in H_\lambda \setminus \{1\} \text{ such that }a_1 h \in L a_2  \,\}.
$$  
For any $\lambda\in\Lambda_6$, there are $a_1,a_2 \in A \cup \{1\}$, $h\in H_\lambda\setminus \{1\}$ and $l_1\in L$ such that $a_1 h=l_1a_2$. 

If $a_1=a_2$, then $l_1\in L\cap a_1 H_\lambda a_1^{-1}$ because of $a_1 h a_2^{-1}=a_1h a_1^{-1}=l_1$. 

If $a_1\ne a_2$, then there exists a finite subset $W(a_1,a_2,\lambda)$ of $G$ such that 
$$
L\cap N_{\{a_2\},1}(a_1 H_\lambda)\subset N_{W(a_1,a_2,\lambda),1}(L\cap a_1 H_\lambda a_1^{-1})
$$ 
by \cite[Lemma 4.22]{M-O-Y1}. 
We hence obtain that $l_2\in L\cap a_1 H_\lambda a_1^{-1}$ and $w \in W(a_1,a_2,\lambda)\cup \{1\}$ such that $l_2 w=l_1$ or $l_2w^{-1}=l_1$. 
Without loss of generality, we assume that $l_2 w=l_1$ (see Figure \ref{fig11}). 

\begin{figure}[top]
\begin{center}
\includegraphics[height=7.5cm]{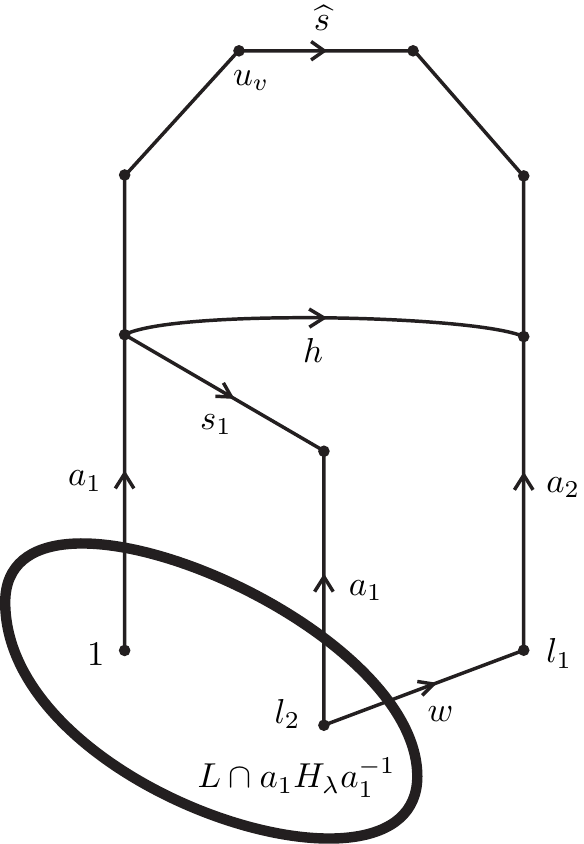}
\end{center}
\caption{An element $l_2$ of $L\cap a_1 H_\lambda a_1^{-1}$}
\label{fig11}
\end{figure}

Set 
$$
\Lambda_7=\{\, \lambda\in\Lambda \mid H_\lambda\in \bigcup_{\stackrel{a_1, a_2 \in A \cup \{1\}}{a_1 \ne a_2}}\scrh(L,a_1,a_2)\,\},
$$ 
where $\scrh(L,a_1,a_2)=\{\, H_\lambda \mid \lambda\in \Lambda, a_1 H_\lambda \cap L a_2 \ne \emptyset\,\}$. 
The set $\Lambda_7$ is finite because $L$ satisfies Condition ($b$) with respect to $\{H_\lambda\}$. 
We take a finite set 
$$
W=\bigcup_{\lambda\in \Lambda_7}\bigcup_{\stackrel{a_1,a_2\in A\cup \{1\}}{a_1\ne a_2}}W(a_1,a_2,\lambda).
$$ 

Let us denote by $r$ the path in $\Gamma(G,X\sqcup (A \cup A^{-1})\sqcup (W \cup W^{-1})\sqcup \calh)$ labeled by
\begin{itemize}
\item $h$ when $a_{1}=a_{2}=1$; 
\item $ha_{2}^{-1}$ (resp.\ $a_{1} h$) when $a_{1}=1$, $a_{2}\ne 1$ (resp.\ when $a_{1} \ne 1$, $a_{2}=1$) and $w=1$;
\item $a_{1}h a_{2}^{-1}$ when $a_{1}=a_{2} \ne 1$, or when $a_{1} \ne a_{2}$, $a_{1}\ne 1$, $a_{2} \ne 1$ and $w=1$;
\item $ha_{2}^{-1}w^{-1}$ (resp.\ $a_{1} h w^{-1}$) when $a_{1}=1$, $a_{2}\ne 1$ (resp.\ when $a_{1}\ne 1$, $a_{2}=1$) and $w \ne 1$; or
\item $a_{1}h a_{2}^{-1}w^{-1}$ otherwise.
\end{itemize} 
The edge of $r$ labeled by $h$ is denoted by $s$. 

Let us denote by $\widehat r$ the path in $\Gamma(G,X\sqcup \Omega\sqcup (A \cup A^{-1})\sqcup (W \cup W^{-1}) \sqcup \calk)$ obtained from $r$ by replaced $s$ with a minimal lift $\widehat{s}$ of $s$. 
We note that $r_{-}, r_{+}, {\widehat r}_-, {\widehat r}_+ \in L\cap a_1H_\lambda a_1^{-1}$. 

Since the natural embedding from $\Gamma(G,X\sqcup\Omega\sqcup\calk)$ into $\Gamma(G,X\sqcup\Omega\sqcup (A \cup A^{-1})\sqcup (W \cup W^{-1}) \sqcup\calk)$ is a quasi-isometry, there are constants $\alpha\geq 1$ and $\beta\geq 0$ satisfying the following: 
For any geodesic $p_1$ in $\Gamma(G,X\sqcup\calh)$, a minimal lift $\widehat{p_1}$ is a locally minimal $(\alpha,\beta)$-quasigeodesic without backtracking in $\Gamma(G,X\sqcup\Omega\sqcup (A \cup A^{-1})\sqcup (W \cup W^{-1}) \sqcup \calk)$ by Lemma \ref{minimallift} and any geodesic in $\Gamma(G,X\sqcup\Omega\sqcup\calk)$ is a locally minimal $(\alpha,\beta)$-quasigeodesic without backtracking in $\Gamma(G,X\sqcup\Omega\sqcup (A \cup A^{-1})\sqcup (W \cup W^{-1}) \sqcup \calk)$. 

The path $\widehat r$ is thereby a locally minimal $(\alpha,3\alpha+\beta+3)$-quasigeodesic without backtracking in $\Gamma(G,X\sqcup\Omega\sqcup (A \cup A^{-1})\sqcup (W \cup W^{-1}) \sqcup \calk)$. 

By \cite[Theorem 2.14]{MP08}, there is a finite subset $T_2$ of $G$ satisfying the following: 
For any locally minimal $(\alpha,3\alpha+\beta+3)$-quasigeodesics $q_1$ and $q_2$ without backtracking in $\Gamma(G,X\sqcup\Omega\sqcup (A \cup A^{-1})\sqcup (W \cup W^{-1})\sqcup\calk)$ with $(q_1)_-=(q_2)_-$ and $(q_1)_+=(q_2)_+$, and any vertex $u_1$ of $q_1$ (resp.\ $v_2$ of $q_2$), there exists a vertex $v_1$ of $q_2$ (resp.\ $u_2$ of $q_1$) with $d_{T_2}(u_1,v_1)\leq 1$ (resp.\ $d_{T_2}(v_2,u_2)\leq 1$). 

Let $r_0$ be a geodesic in $\Gamma(G,X\sqcup\Omega\sqcup\calk)$
\begin{itemize}
\item from $1$ to $l_1$ when $a_1=a_2$, or when $a_1\ne a_2$ and $w=1$; or 
\item from $1$ to $l_2$ when $a_1\ne a_2$ and $w \ne 1$.
\end{itemize}
Note that $(r_0)_-, (r_0)_+\in L\cap a_1 H_\lambda a_1^{-1}$. 

For any $\lambda \in \Lambda_6$ and any vertex $v_1$ of $\widehat r$, there exists a vertex $w_{v_1}$ of $r_0$ such that $d_{T_2\cup Z_A}(v_1,L)\leq d_{T_2}(v_1,w_{v_1})+d_{Z_A}(w_{v_1},L)\leq 2$. 

Take any $l\in L\setminus \{1\}$ and any geodesic $q$ in $\Gamma(G,X\sqcup \Omega\sqcup \calk)$ from $1$ to $l$. 
Let $p$ be a geodesic in $\Gamma(G,X\sqcup\calh)$ from $1$ to $l$ and $\widehat p$ be a minimal lift of $p$. 

We take an arbitrary vertex $v$ of $q$. 
There exists a vertex $u_v$ of $\widehat p$ in $\Gamma(G,X\sqcup\Omega\sqcup (A \cup A^{-1})\sqcup (W \cup W^{-1})\sqcup\calk)$ such that $d_{T_2}(v,u_v)\leq 1$. 
Since the set of vertices of $\widehat{p}$ contains every vertex of $p$, we consider the two cases where $u_v$ is a vertex of $p$, and $u_v$ is a vertex of $\widehat{p}$ but not a vertex of $p$. 

When $u_v$ is a vertex of $p$, the following inequality holds; 
$$
d_{A\cup T_2}(v,L)\leq d_{T_2}(v,u_v)+d_{A}(u_v,L)\leq 2. 
$$

\begin{figure}[top]
\begin{center}
\includegraphics[height=7.5cm]{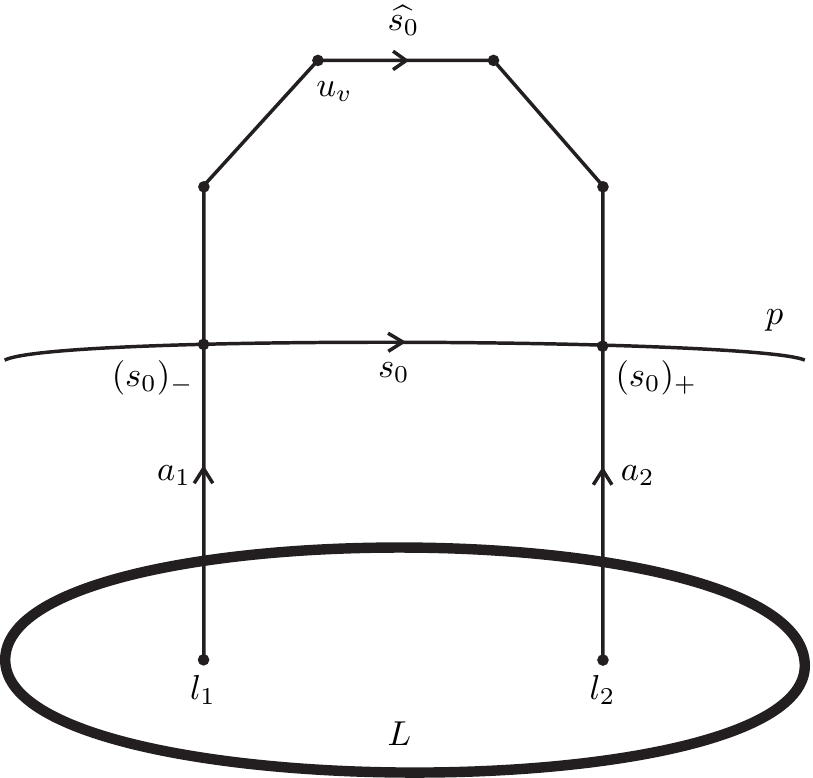}
\end{center}
\caption{An $H_{\lambda_0}$-component $s_0$ of $p$}
\label{fig9}
\end{figure}

We assume that $u_v$ is a vertex of $\widehat{p}$ but not a vertex of $p$.
There is $\lambda_0\in \Lambda$ such that $u_v$ is contained in $\widehat{s_0}$ of an $H_{\lambda_0}$-component $s_0$ of $p$ with $\widehat{s_0} \subset \widehat p$. 
Because $(s_0)_-$ and $(s_0)_+$ are vertices of $p$, we obtain that $a_1,a_2\in A \cup \{1\}$ and $l_1,l_2\in L$ such that $(s_0)_-=l_1a_1$ or $(s_0)_-=l_1a_1^{-1}$, and $(s_0)_+=l_2a_2$ or $(s_0)_+=l_2a_2^{-1}$ (see Figure \ref{fig9}). 
Without loss of generality, we assume that $(s_0)_-=l_1a_1$ and $(s_0)_+=l_2a_2$. 
It then follows that $\lambda_0\in\Lambda_6$. 
We hence obtain that there is a vertex $w_v$ of a geodesic in $\Gamma(G,X\sqcup \Omega\sqcup\calk)$ from $l_1$ to $l_2$ such that $d_{T_2\cup Z_A}(v,L)\leq d_{T_2}(v,u_v)+d_{T_2}(u_v,w_v)+d_{Z_A}(w_v,L)\leq 3$. 
\end{proof}

\begin{Claim}\label{claim7}
The group $L$ satisfies Condition ($b$) with respect to $\{K_{\lambda,\mu}\}$. 
\end{Claim}

\begin{proof}
Let $y_1$ and $y_2$ be elements of $G$ with $Ly_1\cap Ly_2=\emptyset$. 
Set $\scrk(L,y_1,y_2)=\{\, K_{\lambda,\mu} \mid \lambda \in \Lambda, \mu \in M_{\lambda}, y_1K_{\lambda,\mu}\cap Ly_2\ne\emptyset \,\}$. 
We prove that it is a finite set. 

For each $K_{\lambda_0,\mu}\in \scrk(L,y_1,y_2)$, there exist $l\in L$ and $k\in K_{\lambda_0,\mu}$ such that $y_1k=ly_2$. 
Since $L$ satisfies Condition ($b$) with respect to $\{H_\lambda\}$, the set $\scrh(L,y_1,y_2)=\{\,H_{\lambda} \mid \lambda\in \Lambda, y_1H_{\lambda}\cap Ly_2\ne \emptyset\,\}$ is finite. 
We set $\Lambda_8=\{\,\lambda\in \Lambda \mid H_{\lambda}\in\scrh(L,y_1,y_2)\,\}$. 
Because of $K_{\lambda_0,\mu}\subset H_{\lambda_0}$, we obtain that $\lambda_0\in\Lambda_8$. 

We take a minimal finite subset $W_{\lambda_0}$ of $G$ such that $L\cap N_{\{y_2\},1}(y_1H_{\lambda_0})\subset N_{W_{\lambda_0},1}(L\cap y_1H_{\lambda_0}y_1^{-1})$ by \cite[Lemma 4.22]{M-O-Y1}. 
We hence obtain that $l_1\in L\cap y_1H_{\lambda_0}y_1^{-1}$ and $w\in W_{\lambda_0}\cup \{1\}$ with $l=l_1w$ or $l=l_1 w^{-1}$. 
Without loss of generality, we assume that $l=l_1 w$. 
Note that $Ly_1\cap Lwy_2=\emptyset$ because $Ly_1\cap Ly_2=\emptyset$ and $w=l_1^{-1}l\in L$. 
Since $y_1(y_1^{-1}Ly_1\cap H_{\lambda_0})y_1^{-1}=L\cap y_1H_{\lambda_0}y_1^{-1}$ satisfies Condition ($b$) with respect to $\{K_{\lambda,\mu}\}$ by \cite[Lemma 4.16 and Theorem 1.4 (i)]{M-O-Y1} and $l_1\in L\cap y_1H_{\lambda_0}y_1^{-1}$, we obtain that $K_{\lambda_0,\mu}$ is contained in the finite set $\scrk(L\cap y_1H_{\lambda_0}y_1^{-1},y_1,wy_2)$. 
We put $\scrk=\bigcup_{\lambda\in\Lambda_8}\bigcup_{w\in W_{\lambda}\cup \{1\}}\scrk(L\cap y_1H_{\lambda}y_1^{-1},y_1,wy_2)$ and note that $\scrk$ is a finite set. 
The set $\scrk(L,y_1,y_2)$ is finite because $\scrk(L,y_1,y_2)\subset \scrk$. 
\end{proof}

By Claims \ref{claim6} and \ref{claim7}, the group $L$ is quasiconvex relative to $\{K_{\lambda,\mu}\}$ in $G$. 
\end{proof}

\section{Relative hyperbolicity for a group}\label{app}

The aim of this section is to prove Theorem \ref{blowup}. 
We introduce Proposition \ref{rel_Osi_thm1.5} which is used to show that Condition (ii) implies Condition (i) in Theorem \ref{blowup}. 

For a subset $\Lambda_0$ of $\Lambda$, we put $\calh_0=\bigsqcup_{\lambda\in\Lambda_0}\calh_\lambda$ in Proposition \ref{rel_Osi_thm1.5} and Lemma \ref{Osi_lem3.2}. 

\begin{Prop}\label{rel_Osi_thm1.5}
Let $G$ be a group having the finite relative presentation (\ref{representation}) with respect to $\{H_\lambda\}$. 
We suppose that $G$ is hyperbolic relative to $\{H_\lambda\}$. 
Let $Q$ be a subgroup of $G$ and $\Lambda_0$ a subset of $\Lambda$ such that for each $\lambda\in\Lambda_0$, the group $H_\lambda$ is a subgroup of $Q$. 
Then $G$ is hyperbolic relative to $\{H_\lambda\}_{\lambda \in \Lambda \setminus \Lambda_0}\cup \{Q\}$ if the following conditions hold:
\begin{enumerate}
\item[(Q1)] There exists a finite subset $Z$ of $Q$ such that $Q$ is finitely generated by $Z$ relative to $\{H_\lambda\}_{\lambda\in\Lambda_0}$.
\item[(Q2)] There exist constants $C\geq 1$ and $D \geq 0$ such that for each element $q \in Q$, the inequality $|q|_{Z\sqcup\mathcal H_0} \leq C |q|_{X\sqcup \mathcal{H}}+D$ holds.
\item[(Q3)] For each $g \in G\setminus Q$, the group $Q \cap gQg^{-1}$ is finite. 
\end{enumerate} 
\end{Prop}

The proof of this proposition is the same way as in the proof of \cite[Theorem 1.5]{Osi06a} by using Lemmas \ref{Osi_lem3.2}, \ref{Osi_lem3.3} and \ref{Osi_lem3.5} instead of Lemmas 3.2, 3.3 and 3.5 in \cite{Osi06a}. 

Throughout Lemmas \ref{Osi_lem3.2}, \ref{Osi_lem3.3} and \ref{Osi_lem3.5}, $G$ is assumed to be a group having the finite relative presentation (\ref{representation}) with respect to $\{H_\lambda\}$, $Q$ is assumed to be a subgroup of $G$ satisfying Conditions (Q1), (Q2) and (Q3), and without loss of generality, we assume that $Z \subset X$.  

\begin{Lem}\label{Osi_lem3.2}
Let $G$ be a group which is hyperbolic relative to $\{H_\lambda\}$. 
For each constant $\alpha >0$, there exists a constant $A=A(\alpha)>0$ depending on only $\alpha$ and satisfying the following: 
For any $a,b \in Q$ and $f,g \in G$, if they satisfy $\max\{|a|_{Z \sqcup \mathcal{H}_0}, |b|_{Z \sqcup \mathcal{H}_0}\}\geq A$, $\max\{|f|_{X\sqcup \mathcal{H}}, |g|_{X \sqcup \mathcal{H}}\}\leq \alpha$ and $a=fbg$, then $f$ and $g$ are elements of $Q$.  
\end{Lem}

\begin{proof}
Instead of an ordinary finite generating set $Y$ of $Q$ in the proof of \cite[Lemma 3.2]{Osi06a}, we take the finite relative generating set $Z$ of $Q$ with respect to $\{H_\lambda\}_{\lambda\in \Lambda_0}$. 
The lemma is then showed in the same way as in the proof of \cite[Lemma 3.2]{Osi06a}. 
\end{proof}

Let $G$ have the relative presentation (\ref{representation}) with respect to $\{H_\lambda\}$ and let $Q$ be a subgroup of $G$ satisfying Conditions (Q1), (Q2) and (Q3) in Proposition \ref{rel_Osi_thm1.5}. 
We assume that $Z\subset X$ and put $Z_0=X\setminus Z$. 
Take two groups
$F=(*_{\lambda \in \Lambda}H_\lambda)* F(Z)* F(Z_0)$ and $F_Q=(*_{\lambda \in \Lambda \setminus \Lambda_0}H_\lambda)* Q * F(Z_0)$, where $F(Z)$ and $F(Z_0)$ denote free groups generated by $Z$ and $Z_0$, respectively. 
Three canonical homomorphisms are denoted by $\beta \colon F \rightarrow G$, $\varepsilon \colon F \rightarrow F_Q$ and $\gamma \colon F_Q \rightarrow G$ satisfying $\beta=\gamma \circ \varepsilon$. 
We then obtain $\text{Ker}\ \beta$ which is the normal closure of $\calr$ in $F$ and $\text{Ker}\ \gamma$ which is the normal closure of $\varepsilon(\mathcal R)$ in $F_Q$. 
Since both $Z_0$ and $\varepsilon(\mathcal R)$ are finite, the group $G$ is finitely presented relative to $\{H_\lambda\}_{\lambda \in \Lambda \setminus \Lambda_0}\cup \{Q\}$.
To simplify our notation, we put the sets of words $\calh_1=\bigsqcup_{\lambda\in\Lambda\setminus\Lambda_0}\calh_\lambda$ and $\mathcal Q=Z_0\sqcup\mathcal H_1 \sqcup(Q\setminus \{1\})$. 
A relative presentation of $G$ with respect to $\{H_\lambda\}_{\lambda\in \Lambda\setminus \Lambda_0}\cup \{Q\}$ is then 
\begin{equation}\label{repre_Q}
\langle \, Z_0, Q, H_{\lambda}, \lambda \in \Lambda \setminus \Lambda_{0} \mid R=1, R\in \varepsilon(\calr) \,\rangle.\end{equation}
Let $W$ be a word over $\calq$ such that $W$ represents the neutral element of $G$. 
We denote by $\areaq(W)$ the relative area of $W$ with respect to (\ref{repre_Q}). 

By definition, the following lemma holds. 

\begin{Lem}\label{Osi_lem3.3}
Let $G$ be a group which is hyperbolic relative to $\{H_\lambda\}$. 
Let $U$, $V$ and $W$ be words over $\mathcal Q$ and let $T$ be a word over $X\sqcup \mathcal{H}$. 
Assume that $U,V \in \text{Ker}\ \gamma$ and $T \in \text{Ker}\ \beta$. 
Then the following conditions hold: 
\begin{enumerate}
\item[(1)] $\areaq(UV)\leq \areaq(U)+\areaq(V)$.
\item[(2)] $\areaq(W^{-1}VW)=\areaq(V)$.
\item[(3)] $\areaq(\varepsilon(T))\leq \area(T)$.
\end{enumerate}
\end{Lem}

Let $W_1, W_2, W_3, q_1$ and $q_2$ be words over $\calq$ such that both $q_1$ and $q_2$ are not empty words over $Q\setminus \{1\}$ and $W_2$ represents an element of the subgroup $Q$ in $G$. 
A word $W$ over $\mathcal Q$ is said to be \textit{primitive} if $W$ is not decomposed as 
$$
W= W_1q_1W_2q_2W_3.
$$

The following lemma is proved in the same way as in the proof of \cite[Lemma 3.5]{Osi06a}. 

\begin{Lem}\label{Osi_lem3.5}
Let $G$ be a group which is hyperbolic relative to $\{H_\lambda\}$. 
Let $W$ be any primitive word over $\mathcal Q$ with $W \in \mathrm{Ker}\ \gamma$.
Assume that there exists a constant $\kappa >0$ such that $\areaq(W)\leq \kappa || W ||$. 
Then a relative Dehn function of $G$ with respect to $\{H_\lambda\}_{\lambda \in \Lambda\setminus\Lambda_0}\cup \{Q\}$ is linear. 
\end{Lem}

\begin{proof}[Proof of Theorem \ref{blowup}]
We prove that Condition (i) implies Condition (iii), Condition (iii) implies Condition (i), Condition (i) implies Condition (ii), and Condition (i) follows from Condition (ii). 

(i) $\Rightarrow$ (iii)
By Proposition \ref{blowup-a}, for any $\lambda\in \Lambda$, the group $H_\lambda$ is hyperbolic relative to $\{K_{\lambda,\mu}\}_\mu$. 
By Lemma \ref{a1}, there is a finite set $\Lambda_1$ satisfying the assertion. 

(iii) $\Rightarrow$ (i)
Let $X$ be a finite relative generating set of $G$ with respect to $\{H_\lambda\}$. 
For each $\lambda\in\Lambda$, let us denote by $Y_\lambda$ a finite relative generating set of $H_\lambda$ with respect to $\{K_{\lambda,\mu}\}_\mu$. 
If $\lambda\in \Lambda \setminus \Lambda_1$, we put $Y_\lambda=\emptyset$. 
By \cite[Proposition 4.28 (a)]{D-G-O12}, the family $\{H_\lambda\}$ is hyperbolically embedded in $G$ with respect to $X$ in the sense of \cite[Definition 4.25]{D-G-O12}, and for each $\lambda\in\Lambda$, $\{K_{\lambda,\mu}\}_\mu$ is hyperbolically embedded in $H_\lambda$ with respect to $Y_\lambda$ in the sense of \cite[Definition 4.25]{D-G-O12}. 
By a similar argument to the one in the proof of \cite[Proposition 4.35]{D-G-O12}, it is shown that the family $\{K_{\lambda,\mu}\}$ is hyperbolically embedded in $G$ with respect to $X\sqcup \left(\bigsqcup_{\lambda\in\Lambda}Y_\lambda\right)$ in the sense of \cite[Definition 4.25]{D-G-O12}. 
For any $\lambda \in \Lambda$, since the group $H_\lambda$ satisfies Condition ($a$) with respect to $(Y_\lambda, \{K_{\lambda,\mu}\}_\mu)$, there is a finite subset $L_\lambda$ of $M_\lambda$ satisfying the following: 
Whenever any locally minimal cycle without backtracking in $\Gamma(H_\lambda,Y_\lambda \sqcup \calk_\lambda)$ contains a $K_{\lambda,\mu}$-component, the index $(\lambda,\mu)$ is contained in $\{\lambda\} \times L_\lambda$.

It follows that whenever any locally minimal cycle without backtracking in $\Gamma(G, X \sqcup (\bigsqcup_{\lambda \in \Lambda} Y_\lambda) \sqcup \calk)$ contains a $K_{\lambda,\mu}$-component, the index $(\lambda,\mu)$ is contained in $\bigsqcup_{\lambda \in \Lambda_X \cup \Lambda_1} \{\lambda\} \times L_\lambda$.
The group $G$ thus satisfies Condition ($a$) with respect to $(X \sqcup (\bigsqcup_{\lambda \in \Lambda} Y_\lambda), \{K_{\lambda,\mu}\})$. 
Since $X\sqcup \left(\bigsqcup_{\lambda\in\Lambda}Y_\lambda\right)$ is finite and $G$ satisfies Condition ($a$) with respect to $(X\sqcup \left(\bigsqcup_{\lambda\in\Lambda}Y_\lambda\right),\{K_{\lambda,\mu}\})$, it follows from \cite[Theorem 3.1 (iv) and (ii)]{M-O-Y1} that $G$ is hyperbolic relative to $\{K_{\lambda,\mu}\}$. 

(i) $\Rightarrow$ (ii)
Conditions (1) and (2) in (ii) are satisfied when $G$ is hyperbolic relative to $\{H_\lambda\}$. 
Condition (3) in (ii) follows from Proposition \ref{blowup-a}. 

(ii) $\Rightarrow$ (i)
Let $X$ be a finite relative generating set of $G$ with respect to $\{K_{\lambda,\mu}\}$. 
The set $X$ is then a relative generating set of $G$ with respect to $\{H_\lambda\}$. 

We take a finite subset $\Lambda_1$ of $\Lambda$ satisfying the assertion in Lemma \ref{a1}. 
For any $\lambda\in\Lambda\setminus \Lambda_1$, the equality $H_\lambda=\ast_{\mu \in M_\lambda}K_{\lambda,\mu}$ holds. 

Since for any $\lambda\in \Lambda$, the group $H_\lambda$ is quasiconvex relative to $\{K_{\lambda,\mu}\}$ in $G$, the group $H_\lambda$ is undistorted relative to $\{K_{\lambda,\mu}\}$ in $G$ by \cite[Theorem 1.4 (i)]{M-O-Y1}. 
For each $\lambda\in\Lambda_1$, we hence denote by $Y_\lambda$ a finite relative generating set of $H_\lambda$ with respect to $\{K_{\lambda,\mu}\}_\mu$. 
For any $\lambda\in \Lambda\setminus\Lambda_1$, set $Y_\lambda=\emptyset$. 
For any $\lambda\in\Lambda$, the Cayley graph $\Gamma(H_\lambda, Y_\lambda\sqcup \calk_\lambda)$ is quasi-isometrically embedded into $\Gamma(G, X\sqcup \calk)$ by the definition of undistortion \cite[Definition 4.13]{M-O-Y1}. 

We put $M_0=\bigsqcup_{\lambda\in\Lambda_1}(\{\lambda\} \times M_\lambda)$ and $G_1=\langle\, X\sqcup\left(\bigsqcup_{\lambda\in\Lambda_1}\calh_\lambda\right)\,\rangle$. 
The group $G_1$ is also generated by $X\sqcup(\bigsqcup_{\lambda\in\Lambda}Y_\lambda)$ and $\{K_{\lambda,\mu}\}_{(\lambda,\mu)\in M_0}$. 
We denote a relative presentation of $G_1$ with respect to $\{H_\lambda\}_{\lambda\in\Lambda_1}$ by
\begin{equation}\label{pre_G1}
\langle\,X, H_{\lambda}, \lambda \in \Lambda_{1} \mid R=1, R\in \calr_1\,\rangle. 
\end{equation}
We then obtain that $G=G_1\ast (\ast_{\lambda\in\Lambda\setminus\Lambda_1}H_\lambda)$ by the proof of Lemma \ref{a1}. 
The group $G$ thereby has a relative presentation with respect to $\{H_\lambda\}$ denoted by 
\begin{equation}\label{pre_G}
\langle\,X, H_{\lambda}, \lambda \in \Lambda \mid R=1,R\in\calr_1\,\rangle .
\end{equation}

We note that the relative Dehn function of $G_1$ with respect to (\ref{pre_G1}) is also the relative Dehn function of $G$ with respect to (\ref{pre_G}). 
When (\ref{pre_G1}) is a finite relative presentation of $G_1$ with respect to $\{H_\lambda\}_{\lambda\in\Lambda_1}$, the presentation (\ref{pre_G}) is also a finite relative presentation of $G$ with respect to $\{H_\lambda\}$. 
We thus prove that $G_1$ is hyperbolic relative to $\{H_\lambda\}_{\lambda\in\Lambda_1}$. 

Let us set $\Lambda_1=\{1,2,\ldots, n\}$ and prove that $G_1$ is hyperbolic relative to $\{H_{i}\}_{i=1}^{n}$ by induction of $k\in\{1,2,\ldots, n\}$. 
Since $G$ has a finite relative presentation with respect to $\{K_{\lambda,\mu}\}$, the group $G_1$ also has a finite relative presentation with respect to $\{K_{\lambda,\mu}\}_{(\lambda,\mu)\in M_0}$. 
We note that  a relative Dehn function of $G$ with respect to $\{K_{\lambda,\mu}\}$ is also a relative Dehn function of $G_1$ with respect to $\{K_{\lambda,\mu}\}_{(\lambda,\mu)\in M_0}$. 
The group $G_1$ is therefore hyperbolic relative to $\{K_{\lambda,\mu}\}_{(\lambda,\mu)\in M_0}$. 
We suppose that for some $k \in \{1,2,\ldots,n\}$, the group $G_1$ is hyperbolic relative to $\{H_i\}_{i=1}^k\sqcup \{K_{i,\mu}\}_{i\geq k+1}$, where $\{K_{i,\mu}\}_{i\geq k+1}=\bigsqcup_{i=k+1}^n\{K_{i,\mu}\}_{\mu\in M_i}$. 

Put $\calk_{M_0}=\bigsqcup_{\lambda\in\Lambda_1}\calk_\lambda$. 
The Cayley graph $\Gamma(G_1,X\sqcup \left(\bigsqcup_{\lambda\in\Lambda}Y_\lambda\right)\sqcup\calk_{M_0})$ is a subgraph of $\Gamma(G,X\sqcup\left(\bigsqcup_{\lambda\in\Lambda}Y_\lambda\right)\sqcup\calk)$. 
We note that for any $i = 1,2,\ldots, n$, the graph $\Gamma(H_i, Y_i\sqcup \calk_i)$ is quasi-isometrically embedded into $\Gamma(G,X\sqcup\left(\bigsqcup_{\lambda\in\Lambda}Y_\lambda\right)\sqcup\calk)$ and its image is contained in $\Gamma(G_1,X\sqcup \left(\bigsqcup_{\lambda\in\Lambda}Y_\lambda\right)\sqcup\calk_{M_0})$. 
The group $H_{k+1}$ is thus quasiconvex relative to $\{K_{\lambda,\mu}\}_{(\lambda,\mu)\in M_0}$ in $G_1$ by \cite[Theorem 1.4 (i)]{M-O-Y1}. 
By Theorem \ref{qc-updown}, the group $H_{k+1}$ is also quasiconvex relative to $ \{H_i\}_{i=1}^k\sqcup \{K_{i,\mu}\}_{i\geq k+1}$ in $G_1$. 
The graph $\Gamma(H_{k+1},Y_{k+1}\sqcup \calk_{k+1})$ is hence quasi-isometrically embedded into $\Gamma\left(G_1,X\sqcup\left(\bigsqcup_{i=1}^k \calh_i \right)\sqcup \left(\bigsqcup_{i=k+1}^{n} \calk_i \right)\right)$ by \cite[Theorem 1.4 (i)]{M-O-Y1}. 
It follows from Proposition \ref{rel_Osi_thm1.5} that $G_1$ is hyperbolic relative to $ \{H_i\}_{i=1}^{k+1}\sqcup \{K_{i,\mu}\}_{i\geq k+2}$. 
\end{proof}


\end{document}